\documentclass[reqno]{amsart}
\usepackage{personal-commands}
    
\usepackage{mathtools}
\usepackage{adjustbox}
\usepackage{caption}
\usepackage[
			backend=biber, 
			natbib=true,
			bibencoding=utf8,
			style=alphabetic,
			labelalpha=true,
			maxalphanames=5, 
			minalphanames=2,	
			url = false,
]{biblatex} 
\usepackage{mathscinet}
\usepackage{calligra,mathrsfs}
\DeclareMathOperator{\SheafHom}{\mathscr{H}\text{\kern -3pt {\calligra\large om}}\,}

\title[Elliptic arrangements of complex multiplication type]{Elliptic arrangements\\of complex multiplication type}
 
\author{Luca Moci, Roberto Pagaria, Maddalena Pismataro, Alejandro Vargas}

\begin{document}

\begin{abstract}
We provide a natural definition of an elliptic arrangement, extending the classical framework to an elliptic curve $\mathcal{E}$ with complex multiplication.
We analyse the intersections of  elements of the arrangement and their connected components as $\operatorname{End}(\mathcal{E})$-modules.
Furthermore, we prove that the combinatorial data of elliptic arrangements define both an arithmetic matroid and a matroid over the ring $\operatorname{End}(\mathcal{E})$. In this way, we obtain a class of arithmetic matroids that is different from the class of arithmetic matroids realizable via toric arrangements. 
Finally, we show that the Euler characteristic of the complement is an evaluation of the arithmetic Tutte polynomial.


\end{abstract}

\maketitle
\section{Introduction}
    \label{sec:Introduction}

The theory of \textit{hyperplane arrangements} is a bridge between geometry, algebra and combinatorics, with countless applications across these fields and other areas of mathematics.
The leitmotif of arrangement theory consists in the explicit computation of geometric invariants -- such as cohomology and homotopy type -- from combinatorial data associated with the arrangement, such as the poset of layers or the poset of faces.
Foundational results of this kind have been discovered for several decades: the Orlik-Solomon algebra \cite{Arnold,OrlikSolomon}, the Salvetti complex \cite{Salvetti}, the cohomology of wonderful models of hyperplane arrangements \cite{DeConciniProcesiWonderfulModel,FeichtnerYuzvinsky}, among others.

In the last two decades, the researchers' attention has shifted toward more general types of arrangements, where the ambient vector space is replaced with richer and more interesting topological spaces.
Motivated by applications to integral polytopes and box-splines, De Concini and Procesi \cite{DeConciniProcesiToricArrangements, DeConciniProcesiBook} initiated the study of \textit{toric arrangements}, defined as collections of hypertori in a complex torus.
More recently, other generalizations have emerged: \textit{elliptic arrangements}, which are collections of certain type of divisors in a product of elliptic curves $\EC^n$, and \textit{abelian arrangements}, defined as collections of certain subgroups in an ambient space $G^n$ where $G$ is an abelian Lie group.

The connection between topological invariants and combinatorics persists for toric arrangements, although in a subtler way \cite{moc12a, dm13,dAntonioDelucchiSalvettiComplex,CallegaroDelucchi,CDDMP,mp22}, \cite{PagariaTwoExamples}, while it becomes weaker in the elliptic case \cite{lv12,DelucchiPagaria} \cite{PagariaEllipticArrangements} and in the abelian case \cite{bib15,TranYoshinaga,LiuTranYoshinaga,BibbyDelucchi,bpp24}.

The standard definition of elliptic arrangements, used in the literature, is independent of the specific elliptic curve $\EC$.
In fact, one only considers divisors in $\EC^n$ of the form $\ker \phi$ where
\[
\setlength\arraycolsep{0pt}
\phi \colon \begin{array}[t]{c c c} 
          \EC^n &\longrightarrow & \EC \\ 
          (p_1, \dots, p_n) &\mapsto & m_1p_1 + \dots +m_np_n 
         \end{array}
\]
with $m_i \in \mathbb{Z}$.
In fact, in this framework the elliptic curve is treated as a topological group $\EC \simeq S^1 \times S^1$ that admits an Hodge structure.

In this work, we provide a more general and natural definition of an elliptic arrangement, which is also sensible to the choice of the curve $\EC$, thereby introducing number-theoretic aspects into the picture.

More precisely, let us recall that every endomorphism of an elliptic curve $\EC$ arises from multiplication by a complex number  lying in a subring $R_\EC$ of a quadratic number field.
The ring of integers $\ZZ$ is always contained in $R_\EC$.
Generically $R_\EC = \ZZ$, but some elliptic curves,  called of \emph{complex multiplication type (CM)}, admit more endomorphisms.
While all the cited papers assume that $m_i \in\ZZ\subseteq \End \EC$, in the present paper we drop such assumption, building a more general family of hypersurfaces. In other words, we define elliptic arrangements as collections of kernels of \textit{arbitrary} morphisms of abelian varieties $\phi \colon \EC^n \to \EC$.
We extend all previous results to this setting, which allows us to address new phenomena that arise when $\EC$ has complex multiplication.

\medskip

In \Cref{sec:Background}, we recall some basic facts about the endomorphism ring $R_\EC := \End \EC$ of an elliptic curve $\EC$ of complex multiplication type. In particular $R_\EC$ is an order in a quadratic imaginary number field.
In \Cref{sec:Arrangements} we introduce hyperplane, toric, and elliptic arrangements.
While the description of the intersections is trivial in the hyperplane case, simple in the toric one, it becomes significantly more complex for the elliptic arrangements.
A matrix $A \in \Mat_{k,n}(R_\EC)$ defines an elliptic arrangement of $k$ divisors in $\EC^n$.
For any $S \subseteq [k]$, we show that the intersection $\mathcal{A}_S$ of the corresponding divisors fits into a short exact sequence:
  \begin{equation}
  \label{eq:SES_introduction}
0 \to \quotient{\ker \res S \ACC}{\ker \res S A} \to \Arrangement_S \to \tor \coker \res S A \to 0.
  \end{equation}
Here, $\pi_S$ denotes the selection of rows of the matrix $A$ indexed by $S$.
In particular, each intersection $\Arrangement_S$ is an extension between an abelian variety and a finite group.
Moreover, the number of connected components (or \textit{layers}) in $\Arrangement_S$ equals $\# \tor \coker \pi_S \circ A$ (see \Cref{lm:NumberOfLayers}).
We also characterize the abelian varieties that can appear as a layer of an elliptic arrangement in $\EC^n$: they are products of elliptic curves $\EC_i$ that are isogenous to $\EC$, with the conductor of $\End{\EC_i}$ dividing that of $\End \EC$ (see \Cref{lm:descr_conn_comp}).

In \Cref{sec:StructureOverEnd(E)}, we study $\Arrangement_S$ as a $R_\EC$-module,
illustrating through examples how the situation is more subtle than considering $\Arrangement_S$ only as $\ZZ$-module.
In particular, we present \Cref{ex:SequenceNotSplits} where the short exact sequence~\eqref{eq:SES_introduction} does not split as $R_\EC$-modules, and we completely characterize when it splits (see \Cref{prop:DiagramChasing}).
Furthermore, the matrix $A$ induces maps $A_\Lambda \colon \Lambda^n \to \Lambda^k$ and $A_R \colon R^n \to R^k$
and we prove in \Cref{lm:CokernelsCoincide} that 
\[ \Lambda^k / \AL(\Lambda^n) \isom R^k / \AR( R^n) \]
as $\ZZ$-modules, thereby filling a gap in the preprint \cite{bm19} (c.f.\ \Cref{re:}).

In \Cref{sec:ComatroidsOverRings}, we recall the notion of \textit{arithmetic matroid}, introduced in \cite{dm13}, 
and show that every elliptic arrangement $\Arrangement$ gives rise to such a structure.
Specifically, the ground set $[k]$, the rank function $\rk_\Arrangement (S)= \dim_\CC \ker \pi_S \circ A_\CC$, and the multiplicity function $m_\Arrangement(S) = \card \tor{\coker \res S A}$ define an arithmetic matroid (\Cref{thm:4EllipticArr-AritMat}).
The proof relies on a new duality construction for elliptic arrangements (\Cref{lm:coker_conjugate_transpose}).
Interestingly, while arithmetic matroids were originally introduced to study toric arrangements,
we show an elliptic arrangement that corresponds to a non-realizable arithmetic matroid (\Cref{ex:NewRealization}).
We also observe that the so-called GCD property holds when $R_\EC$ is Dedekind (\Cref{lm:GCDpropertyForRDedekind}), but not in general.
We briefly relate elliptic arrangements to (poly-)matroids over $R_\EC$, introduced in \cite{fm16}.
Finally, we compute the Euler characteristic of the complement of an elliptic arrangement (\Cref{pr:EulerCharacteristic}), extending Bibby's work \cite{bib15} to the complex multiplication case.

\section{Background on elliptic curves}
    \label{sec:Background}

\subsection{Notation}
We write $[k]$ for $\aset{1,2, \dots, k}$; 
$\card X$ for cardinality of a set; 
$X \cup i$ for $X \cup \aset{i}$; 
SES for Short Exact Sequence;
$\CoCo X$ for connected components of $X$.

\subsection{Elliptic curves with complex multiplication}
    \label{sub:EllipticCurvesWithComplexMultiplication}
Let $\EC$ be a smooth complex Riemann surface of genus one. 
All such $\EC$ are isomorphic to $\faktor \CC  \Lambda$, 
with $\Lambda$ a lattice generated by $1$ and $\tau \in \CC \setminus \RR$,
and group structure induced by $(\CC, +)$.
If $\End \EC \ne \ZZ$ we say that  $\EC$ has complex multiplication type.
We recall some properties of $\End \EC$ and $\Lambda$;
for details and introductory references, see \cite{Silverman09} or \cite{st15}. 

\subsubsection{The endomorphism ring}
    \label{sub:EndomorphismRing}
We describe a ring $R_\EC$ isomorphic to $\End \EC$.
Denote by $A_\alpha$ the linear map given by multiplication by $\alpha$, i.e.~$z \mapsto \alpha z$.

\begin{lm} 
    \label{lm:EndoRingAsR}
    For an elliptic curve $\EC = \CC / \Lambda$, the ring $\End \EC$ is isomorphic to \[R_\EC = \aset{\alpha \in \Lambda \suchthat \alpha \Lambda \subset \Lambda}\]
    via the map $R_\EC \to \End \EC$ defined by $\alpha \mapsto A_\alpha$.
\end{lm}

\begin{proof}[Sketch of proof]
    One shows that if $f \colon \faktor \CC  \Lambda \to \faktor \CC  \Lambda$ is an endomorphism,
    then an analytic continuation $\hat f \colon \CC \to \CC$ of $f$ around 0 is of the form $\hat f(z) = \alpha z$. 
    Moreover, $\alpha$ must have the property that $\alpha \Lambda \subset \Lambda $ for $f$ being well defined when descending to $\faktor \CC  \Lambda$. 
    See Proposition~6.17 of \cite{st15} for details.
\end{proof}

We write $R$ instead of $R_\EC$ when $\EC$ is clear from the context.

\subsubsection{A quadratic relation for \texorpdfstring{$\tau$}{tau}}
    \label{sub:AQuadraticRelationForTau}
    Assume that $\End \EC \not \isom \ZZ$ and consider $\alpha$ in $R$.
    By \Cref{lm:EndoRingAsR} we have $\alpha \Lambda \subset \Lambda$.
    Equivalently, $\alpha \cdot 1 \in \Lambda$ and  $\alpha \cdot \tau \in \Lambda$.
    Having $\alpha \cdot 1 \in \Lambda$ gives $R \subset \Lambda$,
    so we can write $\alpha = x + y\tau$ with $x, y \in \ZZ$.
    Thus, $\alpha \cdot \tau = x\tau + y\tau^2$.
    Hence, the second condition $\alpha \cdot \tau \in \Lambda$ is true if and only if 
    $y\tau^2$ is in $\Lambda$.
    That is, for some $h, k \in \ZZ$ we have 
    \begin{align} 
        \label{eq:tauisquadratic}
      y\tau^2 = h\tau + k.
    \end{align}
    \begin{lm} 
        \label{lm:RIsAnOrder}
        If $\EC = \CC / \langle 1, \tau \rangle$ has $\End \EC \not \isom \ZZ$, then $\QQ[\tau]$ is a quadratic imaginary number field and $R_\EC$ is an order, i.e.\ a subring of the ring of integers $\calO_{\QQ[\tau]}$ of rank two.
    \end{lm}

    \begin{proof} 
        If $y=0$ in Equation~\eqref{eq:tauisquadratic}, then $\alpha = x$ is an integer. 
        So $\End \EC \not \isom \ZZ$ is equivalent to the existence of a non-zero choice for $y$, which gives the quadratic relation $y\tau^2 - h\tau - k = 0$ for $\tau$.
        Moreover, $\tau$ is non-real because otherwise $\Span_\ZZ \aset{1, \tau}$ would not be a lattice in $\CC$.
        This proves the statement about $\QQ[\tau]$.

        Recall that the ring of integers $\calO$ is the subring of algebraic numbers $\alpha \in \QQ[\tau]$ whose minimal polynomial $f_\alpha$ over $\ZZ$ is monic.
        To see that $\alpha \in R_\EC$ is an algebraic integer,
        multiply Equation~\eqref{eq:tauisquadratic} by $y$ to get the monic relation $(y\tau)^2 -h(y\tau) - yk = 0$, thus $y\tau \in \calO$.
        Since $x \in \ZZ$, we have that $\alpha = x + y\tau$ is in $\calO$ as desired.
    \end{proof}
    
\subsubsection{Quadratic number fields}
    \label{sub:QuadraticNumberFields}
In view of \Cref{lm:RIsAnOrder}, we recall that
every imaginary quadratic number field can be obtained via a unique choice of a square-free positive integer $m$:
let $K_m = \QQ(\sqrt{-m})$ be the field containing $\tau$. 
Consider
\begin{align}
  \label{eq:CasesOmega}
\omega =
\begin{cases}
  \frac{1 + \sqrt{-m}}{2} & \text{ if } m \equiv 3 \mod 4,\\
  \sqrt{-m} & \text{ otherwise,}
\end{cases}
\end{align}
and recall that $\calO_m = \ZZ[\omega]$ is the ring of integers of $K_m$.
In the first case we have the minimal polynomial $f^\ZZ_\omega(\omega) = \omega^2-\omega+m'= 0$ 
where $4m'-1 = m$, and in the second case $f^\ZZ_\omega(\omega) = \omega^2 +m = 0$.  
Since $\tau$ is in  $K_m$, we can choose integers $a,b,c$ with $\gcd(a,b,c) = 1$ such that $\tau = \frac{a+b\omega}{c} \in K_m$.
Calculate
\begin{align}
    \trace A_\tau &= \begin{cases}
    (2a+b)/c & \text{ if } m \equiv 3 \mod 4,\\
  2a/c & \text{ else,}
\end{cases}\\
        \det A_\tau &=\begin{cases}
  (a^2 + ab +b^2m')/c^2 & \text{ if } m \equiv 3 \mod 4,\\
  (a^2 + b^2m)/c^2 & \text{ else;}
\end{cases}
\end{align}
where $\trace A_\tau$ and $\det A_\tau$ are the trace and determinant of the linear map $z \mapsto \tau z$ .
Thus, the minimal polynomial over $\QQ$ of $\tau$ is the characteristic polynomial of $A_\tau$:
\begin{align} 
    \label{eq:MinPolyTau}
    f^\QQ_\tau(\tau) = \tau^2 - \trace A_\tau \tau + \det A_\tau.
\end{align}

\subsubsection{Generators for \texorpdfstring{$R_\EC$}{R}}
    \label{sub:GeneratorsForRE}
    Consider the isomorphism $(x, y) \mapsto x \cdot 1 + y \cdot \tau$ between $\ZZ^2$ and $\Lambda$. 
    Since $R$ is a subgroup of $\Lambda \isom \ZZ^2$ with $1 \in R$, 
    we conclude that $R$ is equal to $\langle 1, N \tau \rangle$ for some non-zero integer~$N$.
    As a consequence of Equation~\eqref{eq:MinPolyTau} and the fact that $N\tau^2$ must have integral coordinates in the basis  $\aset{1, \tau}$ we get that:
\begin{lm} 
    \label{lm:MinimalPolynomialTau}
    We have that $R = \langle 1, N \tau \rangle$ 
    if and only if 
    $N f^\QQ_\tau$ equals the primitive minimal polynomial  $f^\ZZ_\tau$ of $\tau$ over $\ZZ$.
\end{lm}

 

We assume that $N$ is positive, and calculate:

\begin{lm} 
    \label{lm:FormulaForN}
We have that  
\[
   N = c^2/\gcd(c, c^2\det A_\tau).
\]
\end{lm}

\begin{proof} 
    This follows from clearing up denominators from $\trace A_\tau$ and  $\det A_\tau$ in Equation~\eqref{eq:MinPolyTau}. 
    The equalities
    \begin{align*}
        \gcd(c^2, (2a+b)c, a^2+ab+b^2m') &= \gcd(c, a^2+ab+b^2m') & & \text{ if } m\equiv 3 \mod 4, \\
        \gcd(c^2, 2ac, a^2+b^2m) &= \gcd(c, a^2+b^2m) && \text{ otherwise};
    \end{align*}
    hold by the hypothesis $\gcd(a,b,c)=1$ and the claimed result follows.
\end{proof}

\begin{re} 
    \label{re:OrdersOfNumberRings}
    An order in a number field $K$ is a subring whose additive group is finitely generated of rank equal to $[K:\mathbb Q]$.
    Thus, $R_\EC$ is an order in $K_m$.

\end{re}

\section{Arrangements}
    \label{sec:Arrangements}
    A central \textit{abelian arrangement} $\Arrangement$ is a collection of codi\-men\-sion-1 Lie subgroups
    $H_1, \dots, H_k$ in $G^n$ where $G$ is an abelian Lie group. 
    The subgroups $H_i$ are of the form $\ker \phi_i$ for some $\phi_i \in \Hom (G^n,G)$.
    Let $\CM(\Arrangement) = G^n \setminus \bigcup_{i \in [k]} H_i $ be the complement;
    the study of the topology of $\CM(\Arrangement)$ is one of the motivations to study $\Arrangement$.
    We recall the affine and toric cases, 
    to motivate our elliptic setting.

\subsection{Abelian arrangements}
\subsubsection{Hyperplane arrangements}
    \label{subsub:HyperplaneArrangements}
    Here $G = (K,+)$ for some field $K$ and $\phi_i$ is some functional in $\Hom(K^n, K)$, thus $H_i$ is a hyperplane.
A result by Orlik and Solomon establishes that an important combinatorial invariant of $\Arrangement$ is the \emph{lattice of intersections} $\calL(\Arrangement)$, 
i.e.~affine spaces $\aset{\bigcap_{i \in S} H_i \suchthat S \subset [k]}$ partially ordered by reverse inclusion.  
From $\calL(\Arrangement)$ one can reconstruct some topological data of the complement $\CM(\Arrangement)$: 
cohomology when $K = \mathbb C$, 
number of connected components when $K = \RR$,
number of points when $K$ is a finite field $\FF_q$.
See \cite[Theorem~3.5]{dim17} for an introduction.

It turns out that $\calL(\Arrangement)$ as a lattice is semimodular and atomic, thus cryptomorphically is a matroid, realizable over $K$, on the groundset $[k]$; see \cite[Section~1.7]{oxl06}.
Our philosophy is to generalize the matroid approach.

\subsubsection{Toric arrangements}
    \label{subsub:ToricArrangements}
 Here $G = (\CC^*,*)$ is the multiplicative group and $\phi_i$ is some character in $\Hom((\CC^*)^n, \CC^*)\simeq \ZZ^n$, 
so $H_i$ is a union of subtori of codi\-men\-sion~1. 
We will not consider here the case of subtori of arbitrary codimension, 
which has been studied in \cite{mp22}. 
Write $\Arrangement_S$ for the intersection $\bigcap_{i \in S} H_i$ for some $S \subset [k]$.
In the toric and elliptic cases $\Arrangement_S$ is generally not connected.
We call an element of $\CoCo{\Arrangement_S}$ a \emph{layer}, 
i.e. a connected component of $\Arrangement_S$.
All layers of $\CoCo{\Arrangement_S}$ have the same dimension. 
If $\ell_1$ and $\ell_2$ are layers such that $\card \CoCo{\ell_1 \cap \ell_2} \ge 2$,
there is no unique minimal upper bound of $\ell_1$ and $\ell_2$
in the poset $\aset{\bigcap_{i \in S} H_i \suchthat S \subset [k]}$,
so we do not have a lattice. 
Thus, we call $\calC(\Arrangement)$ the \emph{poset of layers}.

This also means that $\calC(\Arrangement)$ cannot correspond to a matroid. 
This is partially fixed in \cite{moc12a} by considering the triple $([k], \rk, m)$
of the functions $\rk S = \codim A_S$ and $m(S) = \card \CoCo{\Arrangement_S}$.
These satisfy some axioms that are christened an \emph{arithmetic matroid} in \cite{dm13, bm14}.
The arithmetic matroid contains enough information to define an arithmetic Tutte polynomial that,
like in the hyperplane case, yields important invariants as evaluations \cite{moc12a}.
For example:
the Poincare polynomial, whose coefficients encode the Betti numbers of the complement $\CM(\Arrangement)$; 
the characteristic polynomial of $\calC(\Arrangement)$, associated with any poset and containing homological information of its order complex. 

\subsubsection{Elliptic arrangements}
    \label{subsub:EllipticArrangements}
    Here $G = \EC$ for some elliptic curve $\EC = \CC / \Lambda$ and $\phi_i$ is in $\Hom(\EC^n, \EC)$.
    By \Cref{sub:EllipticCurvesWithComplexMultiplication} we regard $\phi_i$ as a scalar product $\phi_i(p) = \langle \alpha_i, p \rangle$ with $\alpha_i = (\alpha_{i1}, \dots, \alpha_{in}) \in R^n$.
    Let $A$ be the matrix $(\alpha_{ij})$.
    It gives rise to maps $\AR \colon  R^n \to R^k$, $\AL \colon \Lambda^n \to \Lambda^k$, $\ACC \colon \CC^n \to \CC^k$ and $\AEC \colon \EC^n \to \EC^k$, 
    and for convenience if we omit the subscript then we mean $\AL$.
 
Like in the toric case, we are interested in the number of layers in $\Arrangement_S$ for $S \subset [k]$.
Note that $\ker \alpha_i = \ker \res i  \AEC$ is the $i$-th subvariety in $\Arrangement$.
Thus, $\Arrangement_S = \bigcap_{i \in S} H_i = \ker \res S \AEC$.
We claim the following identity:

\begin{lm} 
    \label{lm:NumberOfLayers}
    Let $\Arrangement$ be an elliptic arrangement in $\EC^n$. 
    For all $S \subset [k]$ the number of layers in $\Arrangement_S$ is:
 \[
     m(S) := \card \CoCo{\Arrangement_S} = \card \tor{\coker \res S \AL}.
 \] 
\end{lm}

For the proof, we find a short exact sequence with the middle term $\Arrangement_S$ such that the sequence splits, giving us a decomposition of $\Arrangement_S$ from which we derive the result. 
Consider the following diagram, where the second and third rows describe the elliptic arrangement
and the first and fourth make an exact sequence, with $\partial$ the map obtained via the snake lemma:    
\begin{figure}[ht!]
 \center

    \begin{tikzpicture}
\matrix[matrix of math nodes,column sep={70pt,between origins},row
sep={30pt,between origins},nodes={asymmetrical rectangle}] (s)
{
    |[name=000]| 0 &|[name=ka]| \ker A &|[name=kb]| \ker \ACC &|[name=kc]| \Arrangement_S \\
|[name=00]| 0 &|[name=A]| \Lambda^n &|[name=B]| \CC^n &|[name=C]| \EC^n &|[name=01]| 0 \\
|[name=02]| 0 &|[name=A']| \Lambda^S &|[name=B']| \CC^S &|[name=C']| \EC^S  &|[name=03]| 0  \\
              &|[name=ca]| \coker A &|[name=cb]| { \coker \ACC}   &|[name=cc]| {\coker \AEC } & |[name=06]|  0. \\
};
\draw[->] 
          (000) edge (ka)
          (00) edge (A)
          (02) edge (A')
          (ka) edge (A)
          (kb) edge (B)
          (kc) edge (C)
          (ka) edge (kb)
          (A) edge (B)
          (A') edge node[auto] {\(\iota \)} (B');
\draw[->] 
          (C') edge (03)
          (cc) edge (06)
          (B) edge node[auto] {\(\)} (C)
          (A') edge (ca)
          (B') edge (C')
          (C) edge (01)
          (cb) edge (cc)
          (B') edge (cb)
          (C') edge (cc);
\draw[->]         
          (A) edge node[auto] {\(\res S A\)} (A')
          (B) edge node[auto] {\(\res S \ACC\)} (B')
          (C) edge node[auto] {\(\res S \AEC \)} (C')
          (kb) edge (kc)
          (kb) edge (kc)
;
\draw[->] 
               (ca) edge node[auto] {\(\bar \iota \)}  (cb)
;
\draw[->,gray,rounded corners] (kc) -| node[auto,text=black,pos=.7]
{\(\partial\)} ($(01.east)+(.5,0)$) |- ($(B)!.35!(B')$) -|
($(02.west)+(-.5,0)$) |- (ca);
\end{tikzpicture}
            \renewcommand{\figurename}{Diagram}
            \caption{}
            \label{dia:SnakeLemma}
            \renewcommand{\figurename}{Figure}
\end{figure}

To simplify notation, for the remainder of this subsection we write $A$ instead of $\res S A$.

We write $\rad \im A = \Lambda^S \cap \im A $ for the \emph{radical} of $\im A$, 
i.e.~the elements in $\Lambda^S$ such that they have a non-zero multiple in $\im A$.
We obtain our desired short exact sequence:
\begin{lm} 
    \label{lm:SESNumberLayers}
    Let $\Arrangement$ be an elliptic arrangement in $\EC^n$. 
    For all $S \subset [k]$ we have the following SES:
 \begin{equation}
          0 \rightarrow \faktor{\ker \ACC}{\ker \AL} \rightarrow \Arrangement_S \rightarrow \im \partial \rightarrow 0.
			\label{SES:NumberOfComponents}
 \end{equation}
Moreover, this sequence splits as $\ZZ$-modules.
\end{lm}

\begin{proof} 
    From the first and fourth row of Diagram~\ref{dia:SnakeLemma} and the snake lemma we readily get the SES~\eqref{SES:NumberOfComponents}.
Note that $\faktor{\ker \ACC}{\ker \AL}$ is a divisible abelian group, 
thus as a $\ZZ$-module it is injective and the sequence splits as $\ZZ$-modules.
\end{proof}

\begin{proof}[Proof of \Cref{lm:NumberOfLayers}] 
The vector space $\ker \ACC \isom \CC^{n-r}$ is connected and so the quotient $\faktor{\ker \ACC}{\ker A}$ is too.
From \Cref{lm:SESNumberLayers} the number of connected components of $\Arrangement_S$ is equal to the one of $\im \partial$.

The result follows from $\rad \im (A_\Lambda)= (\im A_\CC) \cap \Lambda^S$ because 
\[\im \partial = \faktor{(\im A_\CC) \cap \Lambda^S}{\im A_\Lambda} \text{ and } \tor \coker A_\Lambda = \faktor{\rad \im (A_\Lambda)}{\im (A_\Lambda)}.\]
Consider the case of CM elliptic curve, let $K=\operatorname{frac}(R)$ and observe that $K^S \cap \im A_\CC  = \im A_K$, moreover every element in $K$ has a multiple in $\Lambda$, hence $\Lambda^S \cap\im A_K = \rad \im A_\Lambda$.
The proof for non CM elliptic curve is similar and we omit it.

Therefore $\im \partial = \tor{\coker A}$, from which the result follows.
\end{proof}

The behaviour of the SES~\eqref{SES:NumberOfComponents} is more intricate when regarded as $R$-modules.
This is explored in \Cref{sub:ModulesOverR}, including an example where the sequence does not split as $R$-modules.

\subsection{Description of connected components}
Now we focus on the description of the connected component of the identity of $\Arrangement_S$ as an abelian variety.

\begin{de} 
    \label{de:conductor}
 The conductor of $R_\EC$ is $f_\EC  := [\mathcal{O}:R_\EC]= \frac{bc}{\gcd(c,c^2\det A_\tau)}$. 
\end{de}

\begin{lm} Let us fix a complex multiplication elliptic curve $\EC$.
\label{lm:descr_conn_comp}
\begin{enumerate}
    \item Every connected component $\faktor{\ker A_\CC}{\ker A_\Lambda}$ of some elliptic arrangement in $\EC^n$ is a product of elliptic curves $\EC_i$ isogenous to $\EC$ such that $f_{\EC_i} \mid f_\EC$;
    \item Viceversa, every product of elliptic curves $\EC_i$ isogenous to $\EC$ such that $f_{\EC_i} \mid f_\EC$ is a connected component of some elliptic arrangement in $\EC^n$,
\end{enumerate}
\end{lm}
\begin{proof}
\begin{enumerate}
    \item
    Consider a connected component of $A_\EC \colon \EC^n \to \EC^k$, by the previous discussion or by \cite[Theorem 4.4 (c)]{AbVarProdEC} every such morphism arises as a morphism of $R_\EC$-module $A_R^T \colon R^k \to R^n$.
    Choosing a set of generators of $\tor \coker A_R^T$, we construct another morphism $B_\EC \colon \EC^n \to \EC^h$ such that $B_R^T \colon R^h \to R^n$ satisfies $\im B_R^T = \rad \im A_R^T$. This implies $\faktor{\ker A_\CC}{\ker A_\Lambda} = \ker B_\EC$ by \cite[Theorem 4.4 (b)]{AbVarProdEC}.
    The result follows immediately from \cite[Theorem 7.5]{AbVarProdEC} applied to $B_\EC$.
    \item Let $X$ be an abelian variety satisfying the hypothesis, by \cite[Theorem 7.5]{AbVarProdEC} it arises from a torsion-free $R$-module $M$. Choose a free presentation of $M$
    \[ R^k \to R^n \to M \to 0 \]
    by \cite[Theorem 4.4 (b)]{AbVarProdEC} it corresponds to a sequence 
    \[ 0 \to X \to \EC^n \to \EC^k\]
    hence $X$ is a layer of an arrangement of $k$ divisors in $\EC^n$.
\end{enumerate}
\end{proof}

\begin{remark}
    The techniques of \cite{AbVarProdEC,KaniCMelliptic} cannot be applied to the intersection of $\Arrangement _S$ but only to a connected component (layer). Indeed the functor $\SheafHom_R(-,\EC)$ that they consider is not fully faithful on torsion modules.
\end{remark}

\section{Modules over \texorpdfstring{$R_\EC$}{Rℰ}}
    \label{sec:StructureOverEnd(E)}
Our central aim is to describe $\Arrangement_S$.
A good deal of our efforts are dedicated to studying the behaviour of the sequence 
 \[
     \zeta \colon 0 \rightarrow \quotient{\ker \ACC}{\ker A} \rightarrow \Arrangement_S \rightarrow \tor{\coker A} \rightarrow 0
\]
as modules over~$R_\EC$. 
We give some conditions under which $\zeta$ splits, and an example in which it does not.

\subsection{An example of non-splitting}
    \label{sub:AnExampleOfNonSplitting}
    If either $\tor{\coker A}$ is projective or $\quotient{\ker \ACC}{\ker A}$ is injective, 
    we get that $\zeta$ splits.
    The former only happens when $\tor{\coker A}$ is trivial, 
    because projective modules are torsion free.
    The latter offers more hope, as $\quotient{\ker \ACC}{\ker A}$ is divisible, 
    so we are done if we regard it over a principal ideal domain, e.g.~as $\ZZ$-module.
    Unfortunately, $R$ is not necessarily a PID,
    and quotients of $\CC^n$ by a lattice are not necessarily injective as $R$-modules;
    we illustrate this now.

\begin{ex} 
    \label{ex:LambdaNotInjective}
    Let $R = \ZZ[\sqrt{-3}]$ and $\Gamma = \left \langle 1, (1 + \sqrt{-3})/2 \right \rangle$.
    We claim that $\quotient \CC  \Gamma$ is not an injective $R$-module. 
    Consider the map $\AR \colon R^2 \to R$ given by $(x,y) \mapsto 2x + (1 + \sqrt{-3})y$.
    The kernel is generated by 
    \[
        \ker \AR = \left \langle 
        v = \begin{psmallmatrix} 2 \\ -1 + \sqrt{-3}  \end{psmallmatrix},  
        w = \begin{psmallmatrix} 1 + \sqrt{-3} \\ -2\end{psmallmatrix}  \right \rangle.
    \]
    Note that $w = (1 + \sqrt{-3})/2v$, 
    so the map $v \mapsto 1$ shows that $\ker \AR$ is isomorphic to $\Gamma$, 
    thus $\quotient \CC \Gamma$ is isomorphic to $\ker \ACC / \ker \AR$.

    Let $\iota \colon \im \AR \to R$ and $\pi \colon \CC \to \quotient{\CC}{\Gamma}$ be the canonical injection and surjection, respectively.
    We construct an $f \in \Hom_R(\im \AR, \quotient{\CC}{\Gamma})$ that cannot be extended to $R \to \quotient{\CC}{\Gamma} $.
    By the first isomorphism theorem, 
    a map $f$ in $\Hom_R(\im \AR, \quotient{\CC}{\Gamma})$ is induced by 
    a map $R^2 \to \quotient \CC \Gamma$ that vanishes on $\ker \AR$. 
    This gives the following conditions:
    \begin{equation}
    \begin{split} 
        \label{eq:sizygies}
        2 f(2) &\equiv (1 - \sqrt{-3}) f(1+\sqrt{-3})  \mod \Gamma, \\
        (1 + \sqrt{-3}) f(2) &\equiv 2 f(1+ \sqrt{-3})  \mod \Gamma.
    \end{split}       
    \end{equation}
    We take the following values:
    \begin{align*} 
      f(2) &= 0 &
      f(1 + \sqrt{-3}) &= \dfrac{1}{ 1 - \sqrt{-3}}.
    \end{align*}
    Suppose $g \colon R \to \quotient \CC \Gamma$ lifts $f$. 
    For some $\gamma \in \Gamma$ we have that
    \[(1 + \sqrt{-3})g(1) =  g(1 + \sqrt{-3}) = f(1 + \sqrt{-3}) = \frac{1}{ 1 - \sqrt{-3}} + \gamma. \]
    Thus, 
    \[
        2g(1) = 2 \left( \frac{1}{ (1+\sqrt{-3}) (1 - \sqrt{-3})} + \frac{\gamma}  {1 + \sqrt{-3}} \right)  = \frac 1 2 + \frac{2\gamma}{1+\sqrt{-3}}.
    \]
    Finally, a quick calculation on the generators verifies that
    $\frac{2}  {1 + \sqrt{-3}} \Gamma = \Gamma$, hence $g(2) \equiv 1/2 \mod \Gamma$, a contradiction.
    Therefore $\quotient \CC  \Gamma \isom \quotient {\ker \ACC}{\ker \AR}$ is not injective.
\end{ex}

The failure of injectivity of $\quotient{\ker A_\CC}{\ker A}$ sets the stage for the failure of the sequence to split, 
which we verify by considering $A$ in the ambient space $(\quotient{\CC}{\ZZ[\sqrt{-3}]})^2$.

\begin{ex} 
    \label{ex:SequenceNotSplits}
    Let $A$ be the matrix from \Cref{ex:LambdaNotInjective}. 
    We study the arrangement defined by $A$ in the ambient space $(\CC / \Lambda)^2$ with $\Lambda = \ZZ[\sqrt{-3}]$.
    That is, the parameters are $m=3$ and $\tau= \sqrt{-3}$.
    Set $S = \aset{1, 2}$, so $\Arrangement_S = \ker \AEC = \aset{z \in \CC \suchthat \ACC(z) \in \ZZ[\sqrt{-3}]  } / (\ZZ[\sqrt{-3}])^2$.
    By \Cref{lm:FormulaForN} we have $N = 1$, so $R$ equals $\ZZ[\sqrt{-3}]$ as well. 
    Thus, $\ker A = \ker A_R$ and \Cref{ex:LambdaNotInjective} tells us that $\quotient{\ker A_\CC}{\ker A}$ is not injective, so there is a chance that $\zeta$ does not split.
    Over $\ZZ$ we have
    \begin{align}   
        \label{eq:NonRealizableAsToric}
        \AZZ = 
      \begin{pmatrix} 2 & 0 & 1 & -3 \\ 
            0 & 2 & 1 & 1 
          \end{pmatrix}.
    \end{align} 
    Hence, $\card{\tor{\coker A}} = \gcd(2 \times 2\text{ minors of } A) = 2$ and so $\tor \coker A \isom \quotient \ZZ {2\ZZ}$.
    This suggests to take $\zeta$ and look at elements of order~2 in each of the modules.
    Given a group $X$, write $X[2]$ for its 2-torsion, i.e.~elements $x$ such that $2x = 0$.
    First, we have $(\tor \coker A)[2] = \tor \coker A \simeq R/I$ where $I=(2,1+\sqrt{-3})$.
    Since the short exact sequence $\zeta$ splits as $\ZZ$-module, we have
    \[
      \zeta[2] \colon 0 \to \quotient{\ker \ACC}{\ker A}[2] \to \Arrangement_S[2] \to \tor \coker A[2] \to 0
    \] 
    as $R$-modules.
    Since $\quotient{\ker \ACC}{\ker A}$ is an elliptic curve, its order-2 points are generated by $v/2$ and $w/2$, with $v$ and $w$ as in \Cref{ex:LambdaNotInjective}.
    Namely,
   \[  (\quotient{\ker \ACC}{\ker A})[2] = \aset{
        \begin{psmallmatrix} 0 \\   0 \end{psmallmatrix},  
        \begin{psmallmatrix} 1 \\   (-1 + \sqrt{-3})/2  \end{psmallmatrix},  
        \begin{psmallmatrix} (1 + \sqrt{-3})/2 \\ -1\end{psmallmatrix}  
        \begin{psmallmatrix} (3 + \sqrt{-3})/2 \\ (-3 + \sqrt{-3})/2  \end{psmallmatrix}  }
        \simeq \left( \faktor{R}{I} \right) ^2. \]
    We have that $\card \Arrangement_S[2] = (\card \quotient{\ker \ACC}{\ker A}[2]) \cdot (\card \tor \coker A[2]) = 8$ since $\zeta[2]$ is an exact sequence over $\ZZ$.
    We already have four elements coming from the injection of $\quotient{\ker \ACC}{\ker A}$, plus the element $(1/2, 0)$ we had found before, so we compute:
    \begin{align*} 
        (\Arrangement_S)[2] = & \left \{ 
        \begin{psmallmatrix} 0 \\   0 \end{psmallmatrix},  
        \begin{psmallmatrix} 0 \\   (-1 + \sqrt{-3})/2  \end{psmallmatrix},  
        \begin{psmallmatrix} (1 + \sqrt{-3})/2 \\ 0\end{psmallmatrix}  
        \begin{psmallmatrix} (1 + \sqrt{-3})/2 \\ (-1 + \sqrt{-3})/2  \end{psmallmatrix} \right. ,   \\
         & \left.  \phantom{\{} \begin{psmallmatrix} 1/2 \\   0 \end{psmallmatrix},  
        \begin{psmallmatrix} 1/2 \\   (-1 + \sqrt{-3})/2  \end{psmallmatrix},  \begin{psmallmatrix} \sqrt{-3}/2 \\ 0\end{psmallmatrix}, \begin{psmallmatrix} \sqrt{-3}/2 \\ (-1+\sqrt{-3})/2  \end{psmallmatrix} \right \} \simeq \faktor{R}{I} \times \faktor{R}{(2)}.
    \end{align*}
    In particular, $\zeta[2]$ does not split.
    Therefore $\zeta$ does not split, as a splitting of $\zeta$ would give a splitting of $\zeta[2]$.
\end{ex}

\subsection{Splitting of \texorpdfstring{$\Arrangement_S$}{A sub S} }
    \label{sub:DoesTheSESSplit}
    Choose $S \subseteq [k]$ and for brevity write $A$ instead of $A_S$.
    We now relate the splitting of the sequence 
 \[
     \zeta \colon 0 \rightarrow \faktor{\ker \ACC}{\ker A} \rightarrow \Arrangement_S \rightarrow \tor{\coker A} \rightarrow 0
\]
    with the splitting of the sequence
 \[
	 \eta \colon 0 \rightarrow \ker A \rightarrow \Lambda^n \rightarrow \im A \rightarrow 0.
\]
    The left term of $\zeta$ is described by the sequence
 \[
     \mu \colon 0 \rightarrow \ker A \rightarrow \ker \ACC \rightarrow  \faktor{\ker \ACC}{\ker A}  \rightarrow 0,
\]
    and the right one by 
 \[
	 \nu \colon 0 \rightarrow \im A \rightarrow \rad \im A \rightarrow  \tor \coker A   \rightarrow 0.
\]
Recall that $\zeta$ and $\eta$ correspond to classes in the ext groups $\Exti 1 {\tor \coker A} {\quotient{\ker \ACC}{\ker A} }$ and $\Exti 1 {\ker A} {\im A}$, respectively. 
We relate these two groups:

\begin{lm} 
    \label{lm:Diagramone}
    The Diagram~\ref{dia:BigDiagram} of exact sequences commutes.

\begin{figure}[ht!]
 \center
\begin{tikzcd}[scale=0.9,column sep=tiny]
& 0 \arrow[d]                                                                  & {\Hom(\im A, \ker A)} \arrow[d]                      &   \\
0 \arrow[d] \arrow[r]                              & {\Hom(\tor{\coker A}, \quotient {\ker \ACC}{\ker A})} \arrow[d] \arrow[r]   & \Exti 1 {\tor{\coker A}}{\ker A} \arrow[d] \arrow[r] & 0 \\
{\Hom(\rad \im A, \ker \ACC)} \arrow[r] \arrow[d] & {\Hom(\rad \im A, \quotient {\ker \ACC}{\ker A})} \arrow[r] \arrow[d]       & \Exti 1 {\rad \im A}{\ker A} \arrow[d] \arrow[r]     & 0 \\
{\Hom(\im A, \ker \ACC)} \arrow[r] \arrow[d]      & {\Hom(\im A, \quotient {\ker \ACC}{\ker A})} \arrow[r] \arrow[d]            & \Exti 1 {\im A}{\ker A} \arrow[d] \arrow[r]          & 0 \\
0 \arrow[r]                                        & \Exti 1 {\tor{\coker A}}{\quotient {\ker \ACC}{\ker A}} \arrow[r] \arrow[d] & \Exti 2 {\tor{\coker A}}{\ker A} \arrow[d]    \arrow[r]       &  0 \\
& \Exti 1 {\rad \im A }{\quotient {\ker \ACC}{\ker A}} \arrow[r]               & \Exti 2 {\rad \im A}{\ker A}                         &  
\end{tikzcd}
            \renewcommand{\figurename}{Diagram}
            \caption{}
			\label{dia:BigDiagram}
            \renewcommand{\figurename}{Figure}
\end{figure}
\end{lm}
\begin{proof} 
    Combine $\mu$ and $\nu$ using the bifunctoriality of $\Hom(-, -)$.
    Most of the zeros follow from the fact that $\ker \ACC$ is a $\CC$-vector space, 
    thus is an injective $R$-module,
    hence $\Exti 1 - {\ker \ACC}$ is zero.
    Moreover, $\Hom(\tor \coker A, \ker \ACC)$ is zero because $\ker \ACC$ being a $\CC$-vector space has trivial torsion.
\end{proof}
    Looking at the middle term of the 4th row of Diagram~\ref{dia:BigDiagram},
    if there were an $f \in \Hom(\im A, \quotient{\ker \ACC}{\ker A} )$ that lifts both $\zeta$ and $\eta$, 
    we could perform diagram chasing to relate $\zeta$ and $\eta$.
    Since $\ACC$ is a map of vector spaces, there exists a section $s \colon \im \ACC \to \CC^n$ with $\ACC \circ s = \id_{\im \ACC}$.
    Consider $f \colon \im A \to \quotient{\ker \ACC}{\ker A}$ given by
\[
  A\lambda \mapsto s(A\lambda) - \lambda.
\]
    This is well defined because for another $\lambda'$ such that $A\lambda = A\lambda'$ 
    we have that $A(\lambda - \lambda') = 0$, so $\lambda - \lambda'$ is in $\ker A$ and $(s(A\lambda) - \lambda) - (s(A\lambda') - \lambda') \equiv 0 \mod \ker A.$ 
We show that $f$ is mapped on the one hand to $[\zeta]$, and on the other to~$-[\eta]$.

\begin{lm} 
    \label{lm:fMapsToZeta}
In Diagram~\ref{dia:BigDiagram} we have that $f \mapsto [\zeta]$. 
\end{lm}

\begin{proof} 
    Let $Y$ be the pushout of $f \colon \im A \rightarrow \quotient{\ker \ACC}{\ker A}$ and $\iota \colon \im A \rightarrow \rad \im A$.
    By \cite[Lemma~7.28]{rot08} the following diagram commutes
\begin{figure}[ht!] \center
\adjustbox{scale=1}{%
\begin{tikzcd}
0 \arrow[r] & \im A \arrow[r, "\iota"] \arrow[d, "f"]               & \rad \im A \arrow[r] \arrow[d] & \tor \coker A \arrow[r] \arrow[d, equal] & 0 \\
0 \arrow[r] & \quotient{\ker \ACC}{\ker A} \arrow[r] & Y \arrow[r]                    & \tor \coker A \arrow[r]           & 0.
\end{tikzcd}
}
            \renewcommand{\figurename}{Diagram}
            \caption{}
			\label{dia:fMapsToZeta}
            \renewcommand{\figurename}{Figure}
\end{figure}
Applying the functor $\Hom(-, \quotient{\ker \ACC}{\ker A})$ gives that $f$ maps to the class of the bottom row in the corresponding ext group.
Thus, we are done if the bottom row is equivalent to $\zeta$. 
    We deal first with the square on the left. 
    Consider the diagram:
\begin{center}
\begin{tikzcd}[column sep = tiny, row sep = small]
\im A \arrow[r] \arrow[d]                                                 & \rad \im A \arrow[d] \arrow[rdd, "s", bend left] &                  \\
\quotient{\ker \ACC}{\ker A} \arrow[r] \arrow[rrd, "\iota"', bend right] & Y \arrow[rd, "g"]                                     &                  \\
   &              & \Arrangement_S
\end{tikzcd}
\end{center}
    The diagram commutes:
    take an arbitrary element $A\lambda$ in $\im A$. 
    Going right and then down we have $A \lambda \mapsto s(A\lambda)$; down and right gives $A \lambda \mapsto s(A \lambda) - \lambda$.
    The difference of both images is $s(A \lambda) - (s(A \lambda) - \lambda) = \lambda \equiv 0 \mod \Arrangement_S \subset \CC^n/\Lambda^n$.
    Thus, the outer square commutes and by the universal property of the pushout the map $g\colon Y \mapsto \Arrangement_S$ exists.
    We argue that $g$ is an isomorphism.
    
    Recall that $Y$ can be taken equal to $\quotient{\ker \ACC}{\ker A} \oplus \rad \im A$ 
    quotiented by the submodule $\langle (-f(\lambda), \lambda) \mid \lambda \in \im A \rangle$.
    In this presentation $g$ is given by $(z, \lambda) \mapsto z + s(\lambda)$.
    
    Surjectivity of $g$: take an element in $\Arrangement_S$ with representative $z \in \CC^n$. 
    Thus, $\ACC(z) \in \Lambda^S$ and $\lambda := \ACC(z) \in \rad \im A = \Lambda^S \cap \im A_\CC$.
    In particular, $z-s(A_\CC(z)) \in \ker A_\CC$ and 
    \[ g((z-s(A_\CC(z)), A_\CC(z)))= z-s(A_\CC(z)) + s(A_\CC(z)) = z.\]
    
    Injectivity of $g$: 
    take $z \in \ker \ACC$ and $\lambda \in \rad \im A$ and
    suppose that $(z, \lambda) \in \ker g$, i.e.~$z + s(\lambda) = \mu$ for some $\mu$ in $\Lambda^n$.
    Since $z \in \ker \ACC$, we have $\lambda = \ACC \circ s (\lambda) = A(\mu)$.
    So $\lambda$ is in $\im A$; also $f(\lambda) = s(\lambda) - \mu = -z$.
    Thus, $(z, \lambda) = (-f(\lambda), \lambda) \equiv 0$ in $Y$, as desired.
    
    To prove the equivalence it remains to show that $Y \rightarrow \tor \coker A$ equals $\partial \circ g$.
    Given $(z, \lambda)$ as before, the former map sends it to $\lambda$ in $\tor \coker A$.
    The latter map first sends it to $z +s(\lambda)$, 
    and then $\partial$ sends it to $\ACC(z + s(\lambda))$, which equals $\lambda$.
\end{proof}
The above lemma implies that the short exact sequence $\pi^*\zeta$ always splits, where $\pi \colon \rad \im A \to \tor \coker A$. However, we already know this fact because it is equivalent to the existence of the section $s$.

\begin{lm} 
    \label{lm:fMapsToEta}
    In Diagram~\ref{dia:BigDiagram} we have that $f \mapsto -[\eta]$. 
\end{lm}

\begin{proof} 
    Let $X$ be the pullback of $-f : \im A \rightarrow \quotient{\ker \ACC}{\ker A}$ and $\iota : \ker A \rightarrow \ker \ACC$.
    By \cite[Lemma 7.29]{rot08} the following diagram commutes: 
\begin{figure}[ht!]
 \center
\begin{tikzcd}
0 \arrow[r] & \ker A \arrow[r] \arrow[d, equal] & X \arrow[r] \arrow[d] & \im A \arrow[r] \arrow[d, "-f"]               & 0 \\
0 \arrow[r] & \ker A \arrow[r]           & \ker \ACC \arrow[r]  & \quotient{\ker \ACC}{\ker A} \arrow[r] & 0

\end{tikzcd}
    \renewcommand{\figurename}{Diagram}
    \caption{}
	\label{dia:fMapsToEta}
    \renewcommand{\figurename}{Figure}
\end{figure}

    The functor $\Hom(\im A, -)$ maps $-f$ to the class of the bottom row in the corresponding $\Ext^1$.
    Thus, we are done if the top row is equivalent to $\eta$. 
    We deal first with the square on the right. 
    Consider the diagram: 
\begin{center}
\begin{tikzcd}[column sep = tiny, row sep = small]
& \ker \ACC \arrow[r]  & \quotient{\ker \ACC}{\ker A} \\
\hat f                                                                    & X \arrow[r] \arrow[u] & \im A \arrow[u, "-f"]               \\
\Lambda^n \arrow[ru, "h"] \arrow[ruu, bend left] \arrow[rru, "A"', bend right] &                       &                              
\end{tikzcd}
\end{center}
where $\hat f$ sends $\lambda \in \Lambda^n$ to $\lambda - s(A\lambda)$;
this $\hat f$ makes the diagram commute.
Thus, by the universal property of the pullback we have a map $h : \Lambda^n \mapsto X$.

Recall that $X$ can be taken to be the sub-module of $\ker \ACC \oplus \im A$ of pairs $(z, A\lambda)$ such that $z \equiv -f(\lambda) \mod \ker A$, and so $h$ maps $\lambda$ to $(\hat f(\lambda), A\lambda)$.
Suppose the latter pair is $(0,0)$ in $X$, so $A\lambda = 0$, and $0 = \hat f(\lambda) = \lambda - s(A\lambda)  = \lambda$,
which proves injectivity.
On the other hand, given an arbitrary element $(z, A\lambda)$ of $X$, we have  $z \equiv -f(\lambda) \mod \ker A$ and $z = \lambda - s(A\lambda) + \mu$ for some $\mu \in \ker A$.
Therefore $\hat f(\lambda + \mu) = \lambda + \mu - f(A(\lambda + \mu)) = \lambda - s(A\lambda) + \mu = z$, thus $\lambda - \mu$ maps to $(z, A \lambda)$, proving surjectivity.

Lastly, $\lambda \in \ker A$ gets mapped to $(\hat f(z), A\lambda) = (\lambda, 0)$ in $X$, showing that the bottom row is equivalent to  $\eta$, so $-f$ maps to $[\eta]$.
Since $\Ext^1$ is a group, we conclude that $f$ maps to $-[\eta]$.
\end{proof}

    In the following write $\pi \colon \ker \ACC \to \quotient{\ker \ACC}{\ker A}$ and $\iota \colon \im A \to \rad \im A$ for the canonical projection and immersion, respectively.
    The previous three results give:

\begin{prop} 
    \label{prop:DiagramChasing}
    The element $[\eta]$ is in $\im \Exti 1 {\iota} {\ker A}$ if and only if the sequence $\zeta$ splits.
\end{prop}

\begin{proof} 
Notice that by \Cref{lm:fMapsToZeta} and \Cref{lm:fMapsToEta} the image of $[\eta]$ and the one of $[\zeta]$ in $\Ext^2(\tor \coker A, \ker A)$ coincide.
The sequence $\zeta$ splits if and only if the image of $[\zeta]$ in $\Ext^2(\tor \coker A, \ker A)$ is zero.
The latter is equivalent to $\eta \in \im \Exti 1 {\iota} {\ker A}$.
\end{proof}

As a corollary we get sufficient easy conditions for the splitting of $\zeta$.

\begin{cor} 
    \label{cor:RDedekindZetaSplits}
    If $R$ is Dedekind, then the sequence $\zeta$ splits.
\end{cor}
\begin{proof} 
    If $R$ is Dedekind, all $\Ext^2$ groups vanish, so 
    $\Ext^1(\iota, \ker A)$ is surjective, thus $[\eta]$ lifts.
    Alternatively, if we regard the fifth row of Diagram~\ref{dia:BigDiagram} we see that $\Ext^1(\tor \coker A, \quotient{\ker \ACC}{\ker A})$ vanishes when $\Ext^2$ vanishes.
\end{proof}
If the map $A_S \colon \Lambda^n \to \im A_S$ has a section then the extension $\zeta$ is trivial, indeed:
\begin{cor} 
    \label{cor:EtaSplitsZetaSplits}
    If $\eta$ splits, then the sequence $\zeta$ splits.
\end{cor} 

\begin{proof} 
    If $\eta$ splits then $[\eta] = 0$ in $\Ext^1(\im A, \ker A)$ and the zero class always lifts.
\end{proof}

    \subsection{The lattice \texorpdfstring{$\Lambda$}{Λ}}
    \label{sub:ModulesOverR}
As $\ZZ$-module we have that $\Lambda \isom \ZZ^2$, thus it is free.
As $R$-module we have that $\Lambda$ is not free because $R \isom \ZZ^2 \isom \Lambda$ as $\ZZ$-modules, 
$\Lambda \isom R$ would be the only option for freeness as $R$-module, 
but evidently this is not the case. 
Clearly $\Lambda$ is not injective either, since it is not divisible.
We show that $\Lambda$ is projective.

\begin{lm} 
    \label{lm:LambdaIsProjective}
   The lattice $\Lambda$ is a projective $R$-module.
\end{lm}

\begin{proof} 
    
    Recall that $R = \ZZ \oplus N\tau \ZZ$ and that $\Lambda = \ZZ \oplus \tau \ZZ$. 
    Since $\Lambda$ is closed under sums and $R \Lambda \subset \Lambda$, it is an $R$-module in $\Quot R$, 
    i.e.~a fractional $R$-ideal.
    Thus,  $\Lambda$ is projective if and only if it is invertible; see e.g.~\cite[Proposition~4.21]{rot08}. 
    We claim that $N \cdot \Lambda \cdot \sigma(\Lambda) = R$, where $\sigma$ is the automorphism of $\QQ[\sqrt{-m}]$ that sends $\sqrt{-m} \mapsto -\sqrt{-m}$. 
    Indeed, 
    \[
        I = N \cdot \Lambda \cdot \sigma(\Lambda) = N \cdot \langle 1, \tau \rangle \cdot \langle 1,  \bar \tau \rangle 
        = N \cdot \langle 1, \tau, \tau + \bar \tau, \tau \bar \tau \rangle.
    \]
    From the expression on the right, we see that $I \subset R$. 
    Moreover, $I$ contains $N(\tau + \bar \tau) = N \trace A_\tau $, also $N \tau \bar \tau = N \det A_\tau$, and $N$.
    By \Cref{lm:FormulaForN} these are the coefficients of the minimal polynomial of $\tau$ over $\ZZ$, so $\gcd(N \trace A_\tau, N \det A_\tau, N) = 1$, which means $1 \in I$ as desired.
\end{proof}

\subsection{Multiplication by \texorpdfstring{$\alpha \in R$}{α ∈ R}}
    \label{sub:MultiplicationByAlpha}
    Let $A \in \Mat_{k \times n}(R_\EC)$ encode an elliptic arrangement in $\EC^n$.
    Consider the induced maps $\AR \colon R^n \to R^k$ and $\AL \colon \Lambda^n \to \Lambda^k$. 
    As in \Cref{sub:EndomorphismRing}, we have $\tau = (a+b\omega)/c$ with $\gcd(a,b,c) = 1$.
    We will show that $R^k/\im A_R$ and $\Lambda^k/\im A_\Lambda$ are isomorphic as abelian groups.
    First, a motivating example that corresponds to smallest non-trivial case. 

\begin{ex}
    \label{ex:LambdaAlphaRAlpha}
When $n=k=1$ we have a map of the form $A_\alpha$ for some $\alpha = x + yN\tau \in R$. 
We claim that
\[
  \Lambda / \alpha \Lambda \isom R / \alpha R
\]
as $\ZZ$-modules.
For this we write $A_\alpha$ as a matrix $A$ in the basis $\aset{1, \tau}$ of $\Lambda$,  
and as a matrix $\tA$ in the basis $\aset{1, N\tau}$, and we compare their Smith normal form. 
Consider the integers $\normTr= c \trace A_\tau$, $\normDet = c^2 \det A_\tau$, $g= \gcd(c, \normDet)$, $c = gc'$ and $\normDet=g\normDet'$.
So  that $N = c^2/g = g(c')^2$, we get
\[
    A = \begin{pmatrix} x & -y \normDet'\\ 
      y g(c')^2 & x + y \normTr c' 
  \end{pmatrix}
  \quad \quad \quad
      \tA = \begin{pmatrix} x & -y \normDet'g(c')^2 \\ 
      y & x + y \normTr c'
  \end{pmatrix}
\]
Clearly $\det A = \det \tA$. We are done if we prove that the g.c.d. of the entries of $A$ and of $\tA$ coincide.  
By Euclidean algorithm this is equivalent to:
\[
  \gcd\left(x, 
  y\gcd\left(
    g(c')^2,
    \normDet',
    \normTr c'
\right) \right)  = 
\gcd\left(x, 
  y\gcd\left(
    1,
    \normDet'g(c')^2,
    \normTr c'
\right) \right) ,
\]
which is true if 
\begin{align} 
 \gcd\left(
    g(c')^2,
    \normDet',
    \normTr c'
\right)  = 1. 
\end{align}
By \Cref{lm:MinimalPolynomialTau} and \Cref{lm:FormulaForN}, 
these three numbers are the coefficients of the minimal polynomial $f_\tau^\ZZ$ of $\tau$ over the integers,
which implies the coprimality.
\end{ex}

Now we consider the general case. 
Consider the basis $\aset{1, \tau}$ of $\Lambda$, each entry of $\AL$ expands into a  $2\times2$ matrix, as in \Cref{ex:LambdaAlphaRAlpha}, to get a matrix in $\Mat_\ZZ(2k,2n)$ representing $\AL$. 
Likewise, the basis $\aset{1, N\tau}$ gives a matrix in $\Mat_\ZZ(2k,2n)$ representing $\AR$.
By reordering the bases, we can write:
\[
    \AL = \begin{pmatrix} X & -Y \normDet'\\ 
      Y N & X + Y \normTr c' 
  \end{pmatrix}
  \quad \quad \quad
      \AR = \begin{pmatrix} X & -Y \normDet'N \\ 
      Y & X + Y \normTr c'
  \end{pmatrix},
\]
with $X, Y \in \Mat_\ZZ(k,n)$. 
We claim that the cokernels of both matrices are isomorphic: 

\begin{lm} 
    \label{lm:CokernelsCoincide}
   Given $\AR \colon R^n \to R^k$ and $\AL \colon \Lambda^n \to \Lambda^k$ we have that
\[
  \Lambda^k / \AL(\Lambda^n) \isom R^k / \AR( R^n)
\]
as additive groups.
\end{lm}
\begin{proof} 
    We regard $R$ and $\Lambda$ as $\ZZ$-modules and we write $A$ for $\AL$ and $\tA$ for $\AR$.
    By the structure theorem for finitely generated modules over PID, it is enough to prove that for all primes $p \in \ZZ$ the localizations 
    $A_{(p)} \colon (\ZZ_{(p)})^{2n} \to (\ZZ_{(p)})^{2k}$ and $\tA_{(p)} \colon (\ZZ_{(p)})^{2n} \to (\ZZ_{(p)})^{2k}$ 
    have isomorphic cokernels.
    This property is preserved by applying elementary row and columns operations.
    We distinguish three cases:
    \begin{description}
        \item [$p \nmid N$] since $N$ is invertible in $\ZZ_{(p)}$ we can multiply the second column of $\tilde{A}_{(p)}$ by $N^{-1}$ and the second row by $N$ and obtain the matrix $A_{(p)}$.
        \item [$p \nmid \delta'$] since $\normDet'$ is invertible in $\ZZ_{(p)}$ we can multiply the second column of $\tilde{A}_{(p)}$ by $-\normDet^{\prime -1}$ and the second row by $-\normDet'$ and obtain the matrix 
        \[ \begin{pmatrix} X & YN \\ 
      -Y \normDet' & X + Y \normTr c' 
  \end{pmatrix} 
  \sim 
  \begin{pmatrix} X & YN+\frac{\normTr c'}{\normDet'} X \\ 
      -Y \normDet' & X
  \end{pmatrix} \sim \begin{pmatrix} X + Y \normTr c'  & YN \\ 
      -Y \normDet' & X
  \end{pmatrix}.
  \] 
        where firstly we added to the second column $\frac{\normTr c'}{\normDet'}$ times the first one and then to the first row $\frac{\normTr c'}{\normDet'}$ times the second one.
        Finally, by exchanging both rows and columns we obtain the matrix $A_{(p)}$.
        \item [$ p \mid N$ and $p \mid \delta'$]
    Let $I_k$ be the $k \times k$ identity, and $I_n$ the $n\times n$ identity. 
    We modify $A$ and $\tA$ so that the integer $N$ is replaced by another integer $N(s)$ coprime with $p$ and then the result follows from the first case.
    We have for some $s \in \ZZ$:
\begin{align*}
  A_{(p)}  & \sim
    \begin{pmatrix} I_k & 0\\ 
      -s I_k & I_k \end{pmatrix}
    \begin{pmatrix} X & -Y \normDet' \\ 
      YN & X + Y \normTr c'\end{pmatrix}
    \begin{pmatrix} I_n & 0\\ 
      s I_n & I_n \end{pmatrix} \\
      &= \begin{pmatrix} X - Ys\normDet'    & -Y \normDet' \\ 
      Y(N+ s \normTr c'+ s^2 \normDet') & X + Y(\normTr c'+s\normDet' )\end{pmatrix}  \\
  \tA_{(p)}  & \sim
    \begin{pmatrix} I_k & -\normDet' s I_k\\ 
      0 & I_k \end{pmatrix}
    \begin{pmatrix} X & -Y \normDet'N \\ 
      Y & X + Y \normTr c'\end{pmatrix}
    \begin{pmatrix} I_n & \normDet' sI_n\\ 
      0 & I_n \end{pmatrix} \\
      &= \begin{pmatrix} X - Ys\normDet'    & -Y \normDet' (N+s\normTr c'+ s^2\normDet')\\ 
      Y & X + Y(\normTr c'+s\normDet')\end{pmatrix} 
\end{align*}
Write $N(s)$ for $N+s\normTr c'+ s^2\normDet' $.   
We are done if there is a choice of $s$ such that $p \nmid N(s)$.
As in \Cref{ex:LambdaAlphaRAlpha}, we have $\gcd(\normDet', \normTr c', N) = 1$, so $p \nmid \normTr c'$, and we can choose any $s \not \equiv 0 \mod p$.
\qedhere
    \end{description}
\end{proof}

\begin{re} 
    \label{re:}
    \Cref{lm:CokernelsCoincide} is used implicitly in the proof of Theorem~6.1 and Theorem~6.2 of \cite{bm19}. 
    They only provide the justification for the 1-dimensional case, which corresponds to our \Cref{ex:LambdaAlphaRAlpha}.
\end{re}

\section{Arithmetic matroid structure}
    \label{sec:ComatroidsOverRings}

\subsection{Arithmetic matroids}
    \label{sub:ArithmeticMatroids}
    Let $E$ be a finite set. 
    A \textit{matroid} on $E$ is given by a function $\rk \colon \powerset E \to \NN$ that satisfies:
    \begin{enumerate}[label=(r\arabic*)]
        \item \label{item:rk_normalized} $ \rk{\varnothing} =  0$,
        \item \label{item:rk_increasing} $\rk X \le \rk(X \cup i) \le \rk X + 1$ for every $X \subset E$ and $i \in E$,
        \item \label{item:rk_submodular} $\rk(X \cup Y) + \rk(X \cap Y) \le \rk X + \rk Y$ for every $X, Y \subset E$.
    \end{enumerate}
    These axioms are an abstraction of the following example, called the \emph{realizable case}:
    given a list $(v_e)_{e \in E}$ of vectors (indexed by $E$) in some finite dimensional vector space  $V$ over some field $K$, 
    set $\rk S = \dim_K \langle v_e \suchthat e \in S \rangle$ for every subset $S \subset E$. The vectors $(v_e)_{e \in E}$ define an arrangement of hyperplanes on the dual vector space $V^*$; the cohomology of the complement of this arrangement can be described explicitly in terms of the matroid. 
    
    On the other hand, the cohomology of the complement of a \textit{toric arrangement} is not determined by the matroid. Let us recall that a toric arrangement is a family of hypersurfaces in a complex torus $T=(\mathbb C^*)^n$, each hypersurface being  defined by a character $\chi\in Hom(T, \mathbb C^*)\simeq \mathbb Z^n$.
    We set $\rk S$ to be the rank of the submodule $\langle \chi_e \suchthat e \in S \rangle_\ZZ$.
    This satisfies axioms~\ref{item:rk_normalized},~\ref{item:rk_increasing}, and~\ref{item:rk_submodular}. 
    Consider also the number $m(S)$ of connected components in the intersection $\bigcap_{i \in S} H_i$.
    The question of axiomatizing $m(S)$ and studying its properties was addressed in \cite{dm13, bm14, dm18}; we now recall some basic facts.
    
    Denote by $[X,Y]$ the interval $\aset{S \subset E : X \subseteq S \subseteq Y}$ in $(\powerset E, \subseteq)$.
    We say that $[X,Y]$ is a \emph{molecule} if we can write $Y$ as a disjoint union $Y = X \sqcup F \sqcup T$ such that for each $S \in [X,Y]$ we have
    \[      \rk(S) = \rk(X) + \card{(X \cap F)}.    
    \]
    This amounts to saying that, after contracting the elements of $X$, the elements of $T$ become loops and the elements of $F$ become coloops.
    Now, an \emph{arithmetic matroid} is a matroid $(E, \rk)$ endowed with a \textit{multiplicity function} $m \colon \powerset E \to \NN$ such that the following axioms are satisfied:
    \begin{enumerate}[label=(A\arabic*)]
     \item \label{item:axiomA1} For all $S \subset E$ and $i \in E $: if $\rk(S \cup i) = \rk(S)$, 
       then $m(S \cup i)$ divides $m(S)$; 
       otherwise $m(S)$ divides $m(S \cup i)$.
      \item \label{item:axiomA2} If $[X,Y]$ is a molecule then
        \[ m(X) m(Y) = m(X \cup F) m(X \cup T).\]        
    \end{enumerate}
    \begin{enumerate}[label=(P)]
      \item \label{item:axiomP}  If $[X,Y]$ is a molecule, $Y = X \sqcup F \sqcup T$, then the number $ \rho(X,Y) $  given by
        \[
        \rho(X,Y) = (-1)^{|T|} \sum_{S \in [X,Y]} (-1)^{|Y|-|S|}m(S)
        \]
        is greater or equal than~0.
   \end{enumerate}
   In the realizable case, Axioms~\ref{item:axiomA1} and~\ref{item:axiomA2} hold by basic algebraic facts on injections, surjections and sums of modules, while Axiom~\ref{item:axiomP} expresses a count of connected components in a toric arrangement, through inclusion-exclusion.

\subsection{The arithmetic matroid of an elliptic 
 arrangement}
    \label{sub:TheArithmeticMatroidAssociatedToAnArrangement}
We now show that to an elliptic arrangement one can naturally associate an arithmetic matroid.

   \begin{cons} 
       \label{cons:ArithmeticMatroidArrangement}
       Let $\EC$ be an elliptic curve and let $A \in \Mat_{k \times n}(R_\EC)$ encode an elliptic arrangement in $\EC^n$.
       Given a subset $S \subset [k]$ define:
       \begin{align*} 
           \rk_\Arrangement(S) &= \codim \Arrangement_S    \\
           m_\Arrangement(S) &= \card \CoCo{\Arrangement_S}
       \end{align*}
   \end{cons}   
      Our aim is to show that the triple $([k], \rk_\Arrangement, m_\Arrangement)$ of Construction~\ref{cons:ArithmeticMatroidArrangement} is an arithmetic matroid, by proving that axioms~\ref{item:axiomA1},~\ref{item:axiomA2} and~\ref{item:axiomP} hold in this case.

      First, by \Cref{lm:NumberOfLayers} the multiplicity $m(S)$ equals $\tor{\coker A_S}$.
   For convenience, let us write $G_S$ for $\tor \coker A_S$.
   Note that if $X \subset Y$, then the natural projection $\pi \colon \Lambda^Y \to \Lambda^X$ induces a map $\barpi \colon G_Y \to G_X$ with the following properties.

   \begin{lm} 
       \label{lm:ProjectionGS}
       Let
       $S \subseteq [k]$ a set and $i \in [k] \setminus S$ an element. 
       Consider the map $\barpi \colon G_{S \cup i} \to G_S$.
       \begin{enumerate} 
         \item If $\rk(S \cup i) = \rk(S)$, then $\barpi$ is injective. 
         \item If $\rk(S \cup i) \ne \rk(S)$, then $\barpi$ is surjective. 
       \end{enumerate} 
   \end{lm}
  
   \begin{proof} 
           
           Let $e_i$ be the standard basis vector with $1$ in the $i$-th coordinate and zeros everywhere else.
           We observe that there exists a nonzero integer $k$ with $k e_i \in \im A_{S \cup i}$ if and only if $\rk(S \cup i) > \rk(S)$. 
           This is because $\rk(S \cup i) = \rk(S)$ if and only if 
           the $i$-th coordinate is linearly dependent on those indexed by $S$.
\begin{enumerate}
    \item
        let $\barv \in G_{S \cup i}$ be nonzero, and $v \in \Lambda^{S \cup i}$ a representative.
           So $mv \in \im A_{S \cup i}$ for a non-zero $m$.
           If $\barv \in \ker \barpi$, there is $x \in \Lambda^n$ such that $\pi(v) = A_S(x)$.
	   Thus, $v + \lambda e_i \in \im A_{S \cup i}$
           for $\lambda = A_{S \cup i}(x)_i - v_i$;
           also $m \lambda e_i = m(v + \lambda e_i) - mv$ is in $\im A_{S \cup i}$. 
           By the observation $m\lambda = 0$, so $\lambda = 0$ and $\ker \barpi$ is trivial.
    \item Let $\barv \in G_S$, and $v \in \Lambda^{S}$ a representative.
           So $mv \in \im A_{S}$ for a non-zero $m$, 
           thus by a similar argument as before there is $\lambda$ such that 
           $(mv, 0) + \lambda e_i$ is in $\im A_{S \cup i}$.
           Moreover, by the observation let $k \in \ZZ \setminus 0$ such that $k e_i \in \im A_{S \cup i}$.
           Thus, $km(v,0) = k((mv,0)+\lambda e_i) - \lambda (k e_i)$ is in $\im A_{S \cup i}$, 
           so $\overline{(v,0)}$ is a torsion element in $G_{S \cup i}$ and also a lift of $\overline v$,
           proving that $\overline \pi$ is surjective. \qedhere
\end{enumerate}            
   \end{proof}

   \begin{cor} 
       \label{cor:SatisfiesA1}
       The triple $([k], \rk_\Arrangement, m_\Arrangement)$ satisfies Axiom~\ref{item:axiomA1}.
   \end{cor}
   \begin{proof}
        For any $ S \subset [k]$ and $i \in [k] \setminus S $, if $\rk_\Arrangement(S \cup i) = \rk_\Arrangement(S)$ then $G_{S \cup i} \hookrightarrow G_S$ by \Cref{lm:ProjectionGS} and so $m_\Arrangement(S \cup i) \mid m(S)$ by \Cref{lm:NumberOfLayers}.
       Otherwise, $\rk_\Arrangement(S \cup i) = \rk_\Arrangement(S)+1$ and the surjection $G_{S \cup i} \twoheadrightarrow G_S$ implies $m_\Arrangement(S) \mid m_\Arrangement(S \cup i)$.
   \end{proof}
    
Next, if $[X,Y]$ is a molecule with $Y = X \sqcup F \sqcup T$ we can chain the maps from the previous lemma to get a commutative square. 
   \begin{figure}[ht!]
   \center
   \begin{tikzcd}
	   0 \arrow[r] & G_{X \sqcup F \sqcup T} \arrow[d] \arrow[r]
   & G_{X \sqcup F} \arrow[d]  \\
	   0 \arrow[r] & G_{X \sqcup T} \arrow[r] \arrow[d]
		       & G_X \arrow[d] \\
	    & 0 & 0
   \end{tikzcd}
            \renewcommand{\figurename}{Diagram}
            \caption{}
			\label{dia:Projection}
            \renewcommand{\figurename}{Figure}
\end{figure}
   \begin{lm} 
       \label{lm:SatisfiesA2}
       The triple $([k], \rk_\Arrangement, m_\Arrangement)$ satisfies Axiom~\ref{item:axiomA2}.
       That is, if $[X,Y]$ is a molecule with $Y = X \sqcup F \sqcup T$, we have that
   \[ m_\Arrangement(X) \cdot m_\Arrangement(X \sqcup F \sqcup T) = m_\Arrangement(X \sqcup F) \cdot m_\Arrangement(X \sqcup T). \]
   \end{lm}
   
   \begin{proof}
We complete Diagram~\ref{dia:Projection} with cokernels and apply the snake lemma to get Diagram~\ref{dia:MultiplicitySnakeLemma}.
   \begin{figure}[ht!]
   \center
    \begin{tikzpicture}
\matrix[matrix of math nodes,column sep={60pt,between origins},row
sep={40pt,between origins},nodes={asymmetrical rectangle}] (s)
{
|[name=000]| 0 &|[name=ka]| \ker \varphi &|[name=kb]| \ker \psi &|[name=kc]| \ker \bar \psi \\
|[name=00]| 0 &|[name=A]| G_{X \sqcup F \sqcup T} &|[name=B]| G_{X \sqcup F} &|[name=C]|  G_{X \sqcup F }/ G_{X \sqcup F \sqcup T}   &|[name=01]| 0 \\
|[name=02]| 0 &|[name=A']| G_{X \sqcup T} &|[name=B']| G_X &|[name=C']|  G_{X }/ G_{X \sqcup F }   &|[name=03]| 0  \\
              &|[name=ca]| 0 &|[name=cb]| 0   &|[name=cc]| 0 & |[name=06]|  0. \\
};
\draw[->] 
          (000) edge (ka)
          (00) edge (A)
          (02) edge (A')
          (ka) edge (A)
          (kb) edge (B)
          (kc) edge (C)
          (ka) edge (kb)
          (A) edge (B)
          (A') edge node[auto] {\(\)} (B');
\draw[->] 
          (C') edge (03)
          (cc) edge (06)
          (B) edge node[auto] {\(\)} (C)
          (A') edge (ca)
          (B') edge (C')
          (C) edge (01)
          (cb) edge (cc)
          (B') edge (cb)
          (C') edge (cc);
\draw[->]         
          (A) edge node[auto] {\( \varphi \)} (A')
          (B) edge node[auto] {\( \psi\)} (B')
          (C) edge node[auto] {\(\overline{ \psi} \)} (C')
          (kb) edge (kc)
          (kb) edge (kc)
;
\draw[->] 
               (ca) edge (cb)
;
\draw[->,gray,rounded corners] (kc) -| node[auto,text=black,pos=.7]
{\(\partial\)} ($(01.east)+(.5,0)$) |- ($(B)!.35!(B')$) -|
($(02.west)+(-.5,0)$) |- (ca);
\end{tikzpicture}
            \renewcommand{\figurename}{Diagram}
            \caption{}
			\label{dia:MultiplicitySnakeLemma}
            \renewcommand{\figurename}{Figure}
\end{figure}
The result follows from the third column if we prove that $\ker \overline \psi$ is trivial, that is $\ker \varphi$ and $\ker \psi$ are isomorphic.
So we show that $\ker \varphi \to \ker \psi$ is surjective.
If $\bary$ is in $\ker \psi$, there is a representative in $\Lambda^{X \sqcup F}$ of the form $(0, v)$, where the zeros are for the coordinates indexed by $X$. 
There is a non-zero $m$ such that $(0, mv) = A_{X \sqcup F}(x)$ for some $x \in \Lambda^n$.
Since $\rk(X \sqcup T) = \rk(X)$, the coordinates indexed by $T$ are dependent on those indexed by $X$.
Thus, $(0, mv) = A_{X \sqcup F}(x)$ implies $(0, mv, 0) = A_{X \sqcup F \sqcup T}(x)$.
Hence $\overline{(0,v,0)}$ is a torsion element in $G_{X \sqcup F \sqcup T}$ and also the desired lift for $\bary$.
\end{proof}

The remaining Axiom~\ref{item:axiomP} is proved in the next subsection by using duality.

\subsection{Arithmetic matroid duality}
    \label{sub:DualOfAnArithmeticMatroid}
Given a triple $M = (E, \rk, m)$ we define 
the \emph{dual rank function} $\rk^*(S)$ and the \emph{dual multiplicity} $m^*(S)$ as
\begin{align} 
    \label{eq:DualArithmeticMatroid}
\rk^*(S) = \card S - (\rk E - \rk(E \setminus S))  \quad \quad m^*(S) = m(E \setminus S).
\end{align}
Note that $(M^*)^*$ equals $M$ again. 
By \cite[Lemma 2.2]{dm13}, if $M$ is an arithmetic matroid, then 
so is $M^* = (E, \rk^*, m^*)$ and we call it the \emph{dual arithmetic matroid}.
By \cite[Section~2]{bm14} Axiom~\ref{item:axiomP} is equivalent to Axiom~\ref{item:axiomA2} plus:
    \begin{enumerate}[label=(P\arabic*)]
        \item \label{item:axiomP1} For all $X \subset E$, if $Y \in [X, E]$ and $\rk X = \rk Y$, 
	     then $\rho(X,Y) \ge 0$.
         \item \label{item:axiomP2} For all $X \subset E$,  if $Y \in [X, E]$ and $\rk^* X = \rk^* Y$, 
	     then $\rho^*(X,Y) \ge 0$. 
\end{enumerate}
Here $\rho^*$ is the analogous expression for the dual matroid.
Looking at Axiom~\ref{item:axiomP1} in this setting, we must prove for an elliptic arrangement $\calA$ that
\[ 
        \rho(X,Y) = \sum_{S \in [X,Y]} (-1)^{|S|-|X|}m(S)
\]
is non-negative. 
We abuse notation by writing $\CoCo S$ instead of $\CoCo{\Arrangement_S}$, 
i.e.~the set of connected components of $\Arrangement_S = \bigcap_{i \in S} H_i$.
\begin{lm}
    \label{lm:CountingComponents}
Let $\Arrangement$ be an elliptic arrangement, for all $X \subset Y \subseteq [k]$ with $\rk X = \rk Y$, we have
\[ 
  \rho(X,Y) = \card \Big(\CoCo X \setminus \bigcup_{i \in Y \setminus X} \CoCo{X \cup i} \Big).
\]
\end{lm}

\begin{proof} 
    The expression on the right considers all the layers $\ell$ such that $\ell$ is in $\Arrangement_X$ and for any $i \in Y \setminus X$ we have that $\ell$ is not a subset of $H_i = \ker A_i$.
    Note that since $\rk X = \rk Y$,
    taking connected components $\CoCo -$ is an inclusion-reversing operation on the interval $[X,Y]$,
    i.e.~for all $S, T \in [X,Y]$ we have that $S \subset T$ implies that $\CoCo T \supset \CoCo S$. 
    For similar reason we have $\CoCo S \cap \CoCo T = \CoCo{S \cup T}$ for all $S, T \in [X,Y]$. 
    Recall that $m(S) = \card \CoCo S$.
    The result follows from a straightforward inclusion-exclusion count.
\end{proof}
   \begin{cor} 
       \label{cor:SatisfiesP1}
       In the context of Construction~\ref{cons:ArithmeticMatroidArrangement}, the triple $([k], \rk_\Arrangement, m_\Arrangement)$ satisfies Axiom~\ref{item:axiomP1}.
    \end{cor} 

    \begin{proof} 
    From \Cref{lm:CountingComponents} follows that $\rho(X,Y)$ is non-negative, as it counts the cardinality of a set.
\end{proof}
    
Finally, Axiom~\ref{item:axiomP2} would follow from duality if were able to build an elliptic arrangement realizing the dual arithmetic matroid.
We believe this construction to be possible, by developing an analogue of the \textit{generalized toric arrangements} developed
in \cite[Section~4]{dm13},
but we deem it excessively technical for our aims.
Thus, in the next section we approach Axiom~\ref{item:axiomP2} with a weaker construction. 

\subsection{The dual matroid as a minor}
    \label{sub:RealizableDual}
We describe an elliptic arrangement that gives an arithmetic matroid $M_{\Omega \cup \Psi}$
with a minor isomorphic to the dual of~$M$.
Thus, $M_{\Omega \cup \Psi}$ satisfies Axiom~\ref{item:axiomP1}.
Since Axiom~\ref{item:axiomP1} is inherited to minors,
$M^*$ satisfies Axiom~\ref{item:axiomP1}, 
and $(M^*)^* = M$ satisfies Axiom~\ref{item:axiomP2} as desired.

The main ingredient is the elliptic analogue of a toric construction from \cite[Section~3.4]{dm13}.
Firstly, let us recall that $k$ elements $\calP = \aset{p_1, \dots, p_k} \subset \ZZ^n$
give an arithmetic matroid $M_\calP$ as follows: 
for $S \subset \calP$ let $G_S = \rad_{\ZZ^n} \langle p \suchthat p \in S \rangle$.
Consider then $\rk_\calP(S) = \rk \langle p \suchthat p \in S \rangle$ and 
\[m_\calP(S) = \# \tor (\ZZ^n/\langle p \suchthat p \in S \rangle) = \# G_S/ \langle p \suchthat p \in S \rangle.\]
The triple $M_\calP = ([k], \rk_\calP, m_\calP)$ is an arithmetic matroid; see \cite[Section~2.4]{dm13}.

Secondly, we recall the contraction of arithmetic matroids $M = (E, \rk, m)$: the contraction $M/T$ by a set $T \subset E$ is an arithmetic matroid on $E \setminus T$ with rank $r_{M/T}(A) = r(A \cup T) - r(T)$ and multiplicity $m_{M/T}(A) = m(A \cup T)$.

Finally, let $\calQ = \aset{q_1, \dots, q_n} \subset \ZZ^k$ be the columns of the $(k \times n)$-matrix $A$ whose $i$-th row is the vector $p_i$ of $\calP$.
Also let $e_i$ be the $i$-th standard vector of $\ZZ^k$ and $\calB = \aset{e_1, \dots, e_k} \subset \ZZ^k$ the collection of standard vectors.
We consider the matroid $M_{\calB \cup \calQ}$ associated to the matrix $\calB \cup \calQ$ with $n + k$ columns: 

\begin{lm}[Theorem~3.8 from \cite{dm13}] 
    \label{lm:DualMatroidToric}
    Let $\calP$ be a list of elements in $\ZZ^n$, the dual $(M_\calP)^*$ is isomorphic to the contraction $M_{\calB \cup \calQ} / \calQ$.
\end{lm}

    We perform a similar construction for elliptic arrangements. Let $\Arrangement$ be an elliptic arrangement defined by a matrix $A \colon \EC^n \to \EC^k$ and consider the dual homomorphism of abelian varieties $A^H \colon (\EC^\vee)^k \to (\EC^\vee)^n$, where $\EC^\vee$ is the dual elliptic curve and $A^H$ is the conjugate transpose of $A$.
    Consider $\mathcal{B}$ the arrangement given by the matrix
    \[ (\EC^\vee)^k \xrightarrow[\begin{pmatrix}
        I_k\\ A^H
    \end{pmatrix}]{} (\EC^\vee)^{k+n}\]
    Let $T= \{ k+1, \dots, n\}$ be the indexes of the rows of $A^H$.
    \begin{lm}
    \label{lm:DualMatroidElliptic}
        Let $\Arrangement$ be an elliptic arrangement, the dual matroid $M_\Arrangement^*$ is isomorphic to the contraction $M_\mathcal{B}/T$ 
    \end{lm}
    Before the proof of \Cref{lm:DualMatroidElliptic} we need a technical result. 
    \begin{lm}
    \label{lm:coker_conjugate_transpose}
        For any matrix $A \in \Mat_{k,n}(R)$ we have $\tor \coker (A \colon \Lambda^n \to \Lambda^k) \simeq \tor \coker (A^H\colon (\Lambda^\vee)^k \to (\Lambda^\vee)^n)$ as abelian groups.
    \end{lm}
    \begin{proof}
        Observe that $\tor \coker A \simeq \tor \coker A^T$ because their elementary divisors coincides.
        It remains to show that $\tor \coker A \simeq \tor \coker \overline{A}$ where $\overline{A}$ is the complex conjugate of the matrix $A$.
        As in \Cref{sub:MultiplicationByAlpha}, consider $\ZZ$-bases of $\Lambda^n$ and $\Lambda^k$ given by $\{e_i, \tau e_i\}_i$ for $i\leq n$, respectively $i\leq k$.
        The matrices that represent $A$ and $\overline{A}$ in the $\ZZ$-basis are of the form
        \[ A_\ZZ = \begin{pmatrix}
            X & -Y\normDet' \\
            YN & X +Y \normTr c'
        \end{pmatrix}
        \text{ and }
        \overline{A}_\ZZ = \begin{pmatrix}
            X +Y \normTr c' & Y\normDet' \\
            -YN & X
        \end{pmatrix}
        \]
        We show that they have the same Smith normal form, indeed by exchanging rows and columns 
        \[ A_\ZZ \sim \begin{pmatrix}
            X +Y \normTr c' & YN \\
            -Y\normDet' & X
            \end{pmatrix} \sim  
            \begin{pmatrix}
            X +Y \normTr c' & Y \\
            -Y\normDet'N & X
            \end{pmatrix}
            \sim  
            \begin{pmatrix}
            X +Y \normTr c' & -Y\normDet' \\
            YN & X
            \end{pmatrix}
            \sim \overline{A}_\ZZ,
        \]
        where the two middle steps follows from the proof of \Cref{lm:CokernelsCoincide}.
    \end{proof}

    \begin{proof}[Proof of \Cref{lm:DualMatroidElliptic}]
        The statement about underlying matroid is classical, we check only that the two multiplicity functions coincide.
        Let $S \subseteq [k]$, we want to compute 
        \[m_{\mathcal{B}/T}(S) = m_\mathcal{B}(S \cup T) = \# \tor \coker \begin{pmatrix}
        (I_k)_S\\ A^H
    \end{pmatrix}\]
    Using row operation we simplify the matrix  and obtain
    \[ \# \tor \coker \begin{pmatrix}
        (I_k)_S\\ A^H
    \end{pmatrix} =  \# \tor \coker \begin{pmatrix}
        I_s & 0\\ 0 & (A_{[k] \setminus S})^H
    \end{pmatrix}
    =  \# \tor \coker \begin{pmatrix}
         (A_{[k] \setminus S})^H
    \end{pmatrix}\]
    where for simplicity in the middle step we have assumed that $S$ are the first $s$ indices.
    Finally, \Cref{lm:coker_conjugate_transpose} implies $m_{\mathcal{B}/T}(S)= m_\Arrangement ([k] \setminus S)$ and this complete the proof.
    \end{proof}
    
   \begin{cor} 
       \label{cor:SatisfiesP2}
       In the context of Construction~\ref{cons:ArithmeticMatroidArrangement}, the triple $([k], \rk_\Arrangement, m_\Arrangement)$ satisfies Axiom~\ref{item:axiomP2}.
    \end{cor} 

    \begin{proof} 
        By \Cref{cor:SatisfiesA1} the triple $M_{\mathcal{B}} = ([k+n], \rk_\mathcal{B}, m_\mathcal{B})$ satisfies Axiom~\ref{item:axiomP1}.
        This implies that the minor $M_{\mathcal{B}} / T$ satisfies Axiom~\ref{item:axiomP1}.
        To prove the relevant equality, 
        given a molecule $[X, Y]$ in $M_{\mathcal{B}} / T$
        just consider the molecule $[X \cup T, Y \cup T]$ in $M_{\mathcal{B}}$. 
        Now by \Cref{lm:DualMatroidElliptic}, 
        Axiom~\ref{item:axiomP1} holds for $M_\Arrangement^*$, 
        which as observed before proves Axiom~\ref{item:axiomP2} for $M_\Arrangement$ by duality.
\end{proof}
    
\begin{thm}
    \label{thm:4EllipticArr-AritMat}
    The data of ground set $[k]$, of rank function $\rk_\Arrangement$, 
    and of number of connected components 
    $m_\Arrangement(S) = \card \tor{\coker A_S }$ gives rise to an arithmetic matroid.
\end{thm}

\begin{proof}
       Note that $ \codim \Arrangement_S = \rk A_S$, which implies that $([k], \rk_\Arrangement)$ is an underlying matroid. 
       We are left with: \Cref{cor:SatisfiesA1} proving Axiom~\ref{item:axiomA1}; \Cref{lm:SatisfiesA2} proving Axiom~\ref{item:axiomA2};
       Corollaries~\ref{cor:SatisfiesP1}
       and~\ref{cor:SatisfiesP2} proving Axiom~\ref{item:axiomP}.
\end{proof}

\begin{remark}
    Another proof of \Cref{thm:4EllipticArr-AritMat} is suggested by Emanuele Delucchi.
    It consists in checking that the action of $\Lambda^n$ on the periodic hyperplane arrangement $\Arrangement^\upharpoonright$ in $\CC^n$ by translations is \textit{arithmetic}, in the sense of \cite[Definition 3.15]{dr18}.
    Then \Cref{thm:4EllipticArr-AritMat} follows from \cite[Theorem C]{dr18}.
\end{remark}

\subsection{Examples}
We provide two interesting examples:
the first is an arithmetic matroid realizable via elliptic arrangements, but not via toric arrangements.
The second is a variation in which we change the elliptic curve, but the defining matrix is the same. This second arrangement defines an arithmetic matroid realizable as toric arrangement.

   Recall from \cite[Section~3]{dm13} that all arithmetic matroids realizable in the usual sense (that is, via toric arrangements) have the so-called \textit{GCD property}, i.e.~the multiplicity function is defined by the values on the independent sets:
   \[
     m(A) = \gcd(\aset{m(I) \suchthat I \subseteq A \text{ and } |I| = \rk(I) = \rk(A) }).
   \]
   The other implication and the realizability space were studied in \cite{PagariaPaolini}.
   We now show a generalized elliptic arrangement whose arithmetic matroid do not satisfy the GCD property.

\begin{ex} 
    \label{ex:NewRealization}
   Let $\Lambda=\ZZ [\sqrt{-3]}$, $\EC = \CC / \Lambda$, so $\End \EC \isom  \ZZ[\sqrt{-3}]$ as in \Cref{sub:ModulesOverR}. 
   Consider the arrangement $\Arrangement$ associated with the matrix
   \[A= \begin{pmatrix}
       2 \\
       1+\sqrt{-3}
   \end{pmatrix}.\]
   We get:
     \begin{align*}
         m_\Arrangement(1,2) &= \card \ZZ[\sqrt{-3}] / (2, 1 + \sqrt{-3}) = 2, \\
         m_\Arrangement(1) = \card \ZZ[\sqrt {-3}] / 2R = 4, \quad\quad & \quad \quad 
      m_\Arrangement(2) = \card \ZZ[\sqrt{-3}]/(1+\sqrt {-3}) = 4, \\ 
         m_\Arrangement(\varnothing) &= 1. 
    \end{align*}
    This arithmetic matroid does not satisfy the GCD property, since $m(1,2) \ne 4$.

    On the other hand, if we consider $\Gamma=\ZZ[\omega]$ where $\omega$ is a third root of unity, the arrangement $\mathcal{B}$ in $\CC / \Gamma$ defined by the matrix
    \[
    B=\begin{pmatrix}
       2 \\
       1+\sqrt{-3}
   \end{pmatrix}.
    \]
    has multiplicities $m_\mathcal{B}(\varnothing) = 1$ and $m_\mathcal{B}(1) = m_\mathcal{B}(2) = m_\mathcal{B}(1,2) = 4$, 
    which is realizable via a toric arrangement associated to the matrix $B' = \begin{pmatrix}
        4 \\
        4
    \end{pmatrix}$. 
\end{ex}

\subsection{The GCD property}
    \label{sub:TheGCDProperty}
    The proof of the GCD property for toric arrangements
    uses the fact that over a PID every matrix has a Smith normal form.
    This is no longer necessarily true for matrices over $R$.
    Recall that if $R$ is Dedekind, all localizations $R_\frakp$ over all maximal ideals $\frakp \subset R$ are discrete valuation rings, therefore PID. 
    This fact, together with a local-global principle argument, allows us to prove:
\begin{lm} 
    \label{lm:GCDpropertyForRDedekind}
    Let $\Arrangement$ be an elliptic arrangement in $\EC^n$.
    If $R=\End \EC$ is Dedekind, then the arithmetic matroid $M_\Arrangement$ satisfies the GCD property.
\end{lm}
\begin{proof} 
    Let $\frakp \subset R$ be a maximal prime.
    As usual, denote by $R_\frakp = \inv{(R \setminus \frakp)} R $ the localization of $R$ and $M_\frakp$ the $R_\frakp$-module obtained by localizing $M$.
    We use the following facts about the localization of $R$ over a maximal prime ideal $\frakp \subset R$:
    \begin{itemize} 
        \item $\tor M_\frakp = (\tor M)_\frakp$, see e.g.\ \cite[Proposition 3.3]{fm16};
        \item $\coker A_\frakp = (\coker A)_\frakp$, because localization is an exact functor.
        \item If $M$ is a finitely generated torsion module over $R$, then
            $M \simeq \bigoplus_{\frakp} M_\frakp$. 
            The isomorphism follows by the structure theorem for modules over Dedekind domain \cite[Chapter 16, Theorem 22]{DummitFoote} and by the isomorphism $R/\frakp^e \simeq R_\frakp/\frakp^e R_\frakp $.
    \end{itemize}
    Applying these facts we get
    \begin{align}
        \label{eq:GCDpropertyinlocalizations} 
        m_\Arrangement (S) &= \card ( \tor \coker A_S )  \\
        \nonumber        &= \card \left(  \bigoplus _\frakp  \left( \tor \coker A_S \right)_\frakp \right) = \prod_\frakp \card\tor \coker\left( A_S \right)_\frakp 
    \end{align}
    Recall that $R$ is Dedekind if and only if $R$ is Noetherian and all localizations $R_\frakp$ at maximal primes are discrete valuation rings.
    Thus, $m_\frakp(S) = \card \tor \coker\left( A_S \right)_\frakp$
    satisfies the GCD property.
    Since Equation~\eqref{eq:GCDpropertyinlocalizations} expresses $m_\Arrangement$ as the product of $m_\frakp$ over all maximal primes $\frakp \subset R$, we have that $m_\Arrangement$ satisfies the GCD property. 
\end{proof}

\subsection{Matroids over rings}
\label{subsec:matroids_over_rings}
Elliptic arrangements with complex multiplication define naturally matroids over a ring, see \cite{fm16, fm19a} for the definition.

Let $\Arrangement$ be the arrangement defined by a matrix $A \in \Mat_{k,n}(R)$.
Consider the $R$-module $(\Lambda^\vee)^n$ and the submodules 
\[ \im \left(A_S^H \colon (\Lambda^\vee)^S \to (\Lambda^\vee)^n \right) \]
for each $S\subseteq [k]$. This data defines a \textit{realizable polymatroid over} $R$ with the desirable properties
\begin{align*}
    &\codim \Arrangement_S = \rk_R (\coker(A_S^H)), \\
    & \# CC(\Arrangement_S) = \# \tor (\coker(A_S^H)).
\end{align*}
These properties follow from Lemmas~\ref{lm:NumberOfLayers} and~\ref{lm:coker_conjugate_transpose}.
This is not a matroid over $R$ because $\Lambda$ could be a non-free $R$-module: take $k=n=1$ and $A= \begin{pmatrix}
    1
\end{pmatrix}$, the surjection $\Lambda^\vee \to 0$ does not have a $R$-cyclic kernel. 

The construction presented above is the most natural one; however  -- if one is more comfortable working with matroids instead polymatroids -- an alternative construction is available.
Consider the module $R^n$ and the elements $A_R^H e_i$ for each $i \in [k]$, where $e_i$ is the standard base element of $R^k$.
This data defines a \textit{realizable matroid over }$R$ with the properties
\begin{align*}
    &\codim \Arrangement_S = \rk_R (R^n/\langle A_R^H e_i \mid i \in S\rangle ), \\
    & \# CC(\Arrangement_S) = \# \tor (R^n/\langle A_R^H e_i \mid i \in S\rangle ).
\end{align*}
The equalities follow from Lemmas~\ref{lm:NumberOfLayers},~\ref{lm:coker_conjugate_transpose} and~\ref{lm:CokernelsCoincide}.

\subsection{Calculating the Euler characteristic}
    \label{sub:CalculatingTheEulerCharacteristic}
    The work of \cite{bib15} computes the Euler characteristic of the complement $\CM(\Arrangement)$ for abelian arrangements $\Arrangement$ that are defined by a $(k \times n)$-matrix $A$ with integer coefficients.
    That work can be adapted to our setting where $A$ has coefficients in~$R$.
    We do not do a full recap here, but rather indicate the main signposts and the points at which one needs a different argument to prove the result.

    Let $\calX = \EC^n$ and $j \colon \CM(\Arrangement) \hookrightarrow \calX$ the natural embedding.
    We consider the constant sheaf $\underline{\QQ}$ and its higher direct image $R^{\bullet} j_* \underline{\QQ}$, which we aim to relate via the Leray spectral sequence.
    At the second page of such a spectral sequence, we get:
    \[
      E^{p,q}_2 = H^p(X, R^q j_* \underline{\QQ}) \Longrightarrow H^{p,q}(\CM(\Arrangement)).
    \]
    Since elliptic arrangement are locally isomorphic to hyperplane arrangements, for any $x \in \calX$ we have
    \begin{align*} 
        (R^q j_* \underline{\QQ})_x &= \lim_{x \in U} H^q(\inv j(U), \underline{\QQ}) 
        &\simeq \bigoplus_{\substack{W \in \calL_q(\Arrangement)\\ x \in W}} H^{\text{top}}(M(A^{\upharpoonright}_W)) 
    \end{align*}
    Thus $$R^q j_* \underline{\QQ} = \bigoplus_{W \in \calL(\Arrangement)} H^{\text{top}}(M(A^{\upharpoonright}_W)) \otimes \underline{\QQ}_W.$$
    In particular,
    \begin{equation}
    \label{eq:Leray_SS}
                E^{p,q}_2 = H^p(R^q j_* \underline{\QQ}) = \bigoplus_{W \in \calL_q(\Arrangement)} H^{\text{top}}(M(A^\upharpoonright _W)) \otimes H^p(W, \underline{\QQ}).
    \end{equation}
    Using the mixed Hodge structure, we get that the spectral sequence degenerates at the third page $E_3$. Hence $H^\bullet(\CM(\Arrangement), \QQ) = H^\bullet(E^{\bullet, \bullet}_2, d_2)$, both as algebras and as mixed Hodge structures.

    Recall that the arithmetic Tutte polynomial of the arithmetic matroid is 
    \[T_{M}(x,y)= \sum_{S \subseteq [n]} m(A)(x-1)^{\rk[n] -\rk S} (y-1)^{\lvert S \rvert - \rk S}\]
    and the characteristic polynomial of the lattice is
    \[\chi_\calL (t) = \sum_{W \in \calL} \mu_{\calL}(\hat{0}, W) t^{n-\rk W}.\]
    We get the following result.

    \begin{thm}
 
        \label{pr:EulerCharacteristic}
    Let $\Arrangement$ be an elliptic arrangement of rank $n$. The Euler characteristic of the complement is 
     \[
e_{\CM(\Arrangement)} = (-1)^n T_{M_\Arrangement}\left( 1,0 \right) = \chi_{\calL(\Arrangement)} \left(0 \right) = (-1)^n \sum_{P \in \calL_n(\Arrangement)} \operatorname{nbc} (P).
     \]
where $\operatorname{nbc} (P)$ is the number of no broken circuits contained in the flat $P$.
\end{thm}
    
    \begin{proof} 
        The proof goes exactly as the one in \cite{bib15}, which uses a theorem of \cite{moc12a}. We provide a sketch of the proof in our case.
        Let us start by considering the bigraded Poincaré polynomial of $E_2^{\bullet,\bullet}$, described in Equation~\eqref{eq:Leray_SS}:
        \[
         P_{E_2}(t,s) = \sum_{p,q} \dim (E_2^{p,q}) t^p s^q =  \sum_{W \in \calL(\Arrangement)} (-1)^{\rk W} \mu_{\calL (\Arrangement)}(\hat{0}, W) (1+t)^{2(n-\rk W)} s^{\rk W} 
     \]
     where in the second equality we used that the cohomology of an hyperplane arrangement has rank equal to the number of no-broken circuits that is equal to $(-1)^{\rk W} \mu_{\calL (\Arrangement)}(\hat{0}, W)$.
     Recall from \cite[Theorem 5.6]{moc12a} that
     \[ \chi_{\calL(\Arrangement)} (t) = (-1)^n T_{M(\Arrangement)} (1-t,0).\]
     Therefore, 
     \[P_{E_2}(t,s)=(-s)^n\chi_{\calL(\Arrangement)} \left(\frac{(1+t)^2}{-s} \right) = s^n T_{M_\Arrangement}\left( 1+ \frac{(1+t)^2}{s},0 \right) \]
     Evaluating at $s=-t^2$ we obtain 
     \begin{align*}
         P_{\CM(\Arrangement)}(t,-t^2) &=P_{E_2}(t,-t^2) =t^{2n}\chi_{\calL(\Arrangement)} \left(\frac{(1+t)^2}{t^2} \right) \\
         &= (-t^2)^n T_{M_\Arrangement}\left( 1- \frac{(1+t)^2}{t^2},0 \right)
     \end{align*}
     Finally, setting $t=-1$ we have
     \[ e_{\CM(\Arrangement)} = (-1)^n T_{M_\Arrangement}\left( 1,0 \right) = \chi_{\calL(\Arrangement)} \left(0 \right) = (-1)^n \sum_{P \in \calL_n(\Arrangement)} \operatorname{nbc} (P). \]
     This completes the proof.
    \end{proof}
   
\section*{Acknowledgments}
The authors are members of the project PRIN 2022 “ALgebraic and TOPological combinatorics (ALTOP)” CUP J53D23003660006.
The second author is partially supported by INdAM - GNSAGA Project CUP E53C23001670001.

\printbibliography

\end{document}